\documentclass[11pt,reqno]{amsart}
\usepackage{mathtools,amsthm,amsmath,amsfonts,amssymb,tikz-cd,url,enumitem, graphicx,mathrsfs,bbm,bm}
\usepackage[symbol]{footmisc}
\tikzcdset{scale cd/.style={every label/.append style={scale=#1},
    cells={nodes={scale=#1}}}}
\usepackage{setspace}
\usepackage[utf8]{inputenc}
\usepackage{CJKutf8}
\counterwithin{equation}{section}
\usepackage[pdfencoding=auto,psdextra,colorlinks=true,linkcolor=blue,citecolor=blue,urlcolor = orange]{hyperref}
\usepackage{kantlipsum}
\allowdisplaybreaks
\makeatletter
\DeclareFontFamily{OMX}{MnSymbolE}{}
\DeclareSymbolFont{MnLargeSymbols}{OMX}{MnSymbolE}{m}{n}
\SetSymbolFont{MnLargeSymbols}{bold}{OMX}{MnSymbolE}{b}{n}
\DeclareFontShape{OMX}{MnSymbolE}{m}{n}{
    <-6>  MnSymbolE5
   <6-7>  MnSymbolE6
   <7-8>  MnSymbolE7
   <8-9>  MnSymbolE8
   <9-10> MnSymbolE9
  <10-12> MnSymbolE10
  <12->   MnSymbolE12
}{}
\DeclareFontShape{OMX}{MnSymbolE}{b}{n}{
    <-6>  MnSymbolE-Bold5
   <6-7>  MnSymbolE-Bold6
   <7-8>  MnSymbolE-Bold7
   <8-9>  MnSymbolE-Bold8
   <9-10> MnSymbolE-Bold9
  <10-12> MnSymbolE-Bold10
  <12->   MnSymbolE-Bold12
}{}

\let\llangle\@undefined
\let\rrangle\@undefined
\DeclareMathDelimiter{\llangle}{\mathopen}%
                     {MnLargeSymbols}{'164}{MnLargeSymbols}{'164}
\DeclareMathDelimiter{\rrangle}{\mathclose}%
                     {MnLargeSymbols}{'171}{MnLargeSymbols}{'171}
\makeatother
\newtheorem{theorem}{Theorem}[section]
\newtheorem{definition}[theorem]{Definition}
\newtheorem{example}[theorem]{Example}
\newtheorem{lemma}[theorem]{Lemma}
\newtheorem{remark}[theorem]{Remark}
\newtheorem{proposition}[theorem]{Proposition}
\newtheorem{corollary}[theorem]{Corollary}

\newtheorem{notation}[theorem]{Notation}
\newtheorem{assumption}[theorem]{Assumption}
\newcommand{\citep}[1]{\cite{#1}}

\newcommand{\crit}{{\rm crit}}

\newcommand{\dvol}{\text{\normalfont dvol}}

\newcommand{\image}{\text{\normalfont Im}}

\newcommand{\supp}{\text{\normalfont supp}}

\newcommand{\graded}{\mathcal{G}}
\newcommand{\sgn}{\text{\normalfont sgn}}
\newcommand{\adjoint}{\text{\normalfont Ad}}
\newcommand{\calo}{\mathcal{O}}
\newcommand{\calm}{\mathcal{M}}

\newcommand\abs[1]{\left\lvert#1\right\rvert}
\newcommand{\lrcornerwiththings}{\lrcorner\hspace{+0.5mm}}
\newcommand{\professor}{\text{Prof.\ }\hspace{-0.03125mm}}
\newcommand{\doctor}{\text{Dr.\ }\hspace{-0.03125mm}}

\begin{document}
\title[Analytic topological invariant Thom-Smale]{Analytic and topological realizations of the invariant Thom-Smale complex}
\author{Hao Zhuang}
\address{
Department of Mathematics, Washington University in St. Louis, St. Louis, MO 63130, US.}
\curraddr{Beijing International Center for Mathematical Research, Peking University, Beijing 100871, China}
\email{hzhuang@pku.edu.cn}

\begin{abstract}
For a Morse function, its associated Thom-Smale cochain complex admits an analytic realization initiated by Witten. However, due to the unboundedness of the eigenvalues of the deformed Hodge Laplacian along the critical submanifold, the analytic realization of the Thom-Smale complex associated with a Morse-Bott function is a long time open question. In this paper, we give an analytic realization in a case where a compact connected Lie group $G$ acts on a closed oriented manifold $M$. Our construction is not repeating the $G$-equivariant complex, but is actually a $G$-invariant complex computing the de Rham cohomology of $M$. The $G$-invariance is the key to resolve the unboundedness of the eigenvalues along the critical orbits. First, we fix a $G$-invariant Morse-Bott function $f$ on $M$ whose critical set consists of finitely many $G$-orbits, and whose Hessian in the normal direction of each critical orbit is nondegenerate. Second, on the topological side, we simplify the topological Thom-Smale cochain complex associated with $f$ into a $G$-invariant version given by $G$-invariant forms on critical orbits. Third, on the analytic side, we construct the $G$-invariant Witten instanton cochain complex after restricting the deformed Hodge Laplacian on $G$-invariant forms on $M$. As the main results, we prove that both $G$-invariant cochain complexes compute the Betti numbers of $M$, and that there is a cochain isomorphism between these two complexes. Thus, the $G$-invariant Witten instanton cochain complex is the analytic realization that we need.
\end{abstract}


\maketitle
\tableofcontents

\section{Introduction}
In 1982, Witten \cite{witten} suggested an analytic approach to construct the Thom-Smale complex associated with a Morse function satisfying the condition of transversality. In his idea, the de Rham $d$ is replaced by $d_{T} = e^{-Tf}de^{Tf}$, where $f$ is the Morse function, and $T$ is a parameter. Let $T\to+\infty$, the kernel of the deformed Hodge Laplacian $\left(d_T+d_T^*\right)^2$ localizes around the critical set of $f$. This phenomenon hints a possible analytic construction of the Thom-Smale complex using the eigenvalues and eigenforms of the deformed Hodge Laplacian.
The construction was rigorously formulated by Helffer-Sj\"ostrand \cite{Helffer01011985} and Bismut-Zhang \cite{Bismut1994, wittendeformationweipingzhang}. The noncompact case with bounded geometry \cite{dai_yan_2021} and the case for manifold with boundary \cite{lu2017thom} were also formulated in the past several years. 

A straightforward question is: Can we give a similar analytic construction of the Thom-Smale cochain complex when $f$ is a Morse-Bott function, i.e., when the critical set of $f$ is degenerate. Unfortunately, this task in general meets evident systematic obstacles. For general Morse-Bott functions, one major obstacle is caused by the impossibility to find a constant $\alpha$ such that when $T$ is sufficiently large, all the eigenvalues of the Hodge Laplacian on the critical set are $< \alpha$, and all the non-zero eigenvalues of the harmonic oscillator on the normal direction of the critical set are $\gg \alpha$. 

Actually, if we do not build the analytic Thom-Smale complex but just need the Morse-Bott inequalities, the analytic approach always works (See e.g. \cite{bismutcomplicated, BRAVERMAN_FARBER_1997, helffersbottineq, wenlumorsebottineq, ma2007holomorphic, mathai1997equivariant}) with several adjustments in certain cases. This is partly because for the eigenvalues of the deformed Hodge Laplacian on the critical set, we only need those sufficiently small ones to prove the inequalities, and then the constant $\alpha$ always arises. 
Thus, taking the obstacle another way, if in a special case we can separate eigenvalues by a constant $\alpha$, we can then possibly give an analytic construction.

In this paper, we present such a special case under the following assumptions. 
\begin{assumption}\normalfont
    Let $M$ be an oriented closed $m$-dimensional manifold, and $G$ be a connected compact Lie group acting smoothly on $M$. Then, we assume that $M$ carries a $G$-invariant metric $\langle\cdot,\cdot\rangle$ and a smooth $G$-invariant function $f: M\to\mathbb{R}$ satisfying: 
\begin{enumerate}[label = (a\arabic*)]
    \item The critical set of $f$ consists of finitely many $G$-orbits. Also, along the normal direction of the critical set of $f$, the Hessian of $f$ is nondegenerate. We denote the critical set of $f$ by $\crit(f)$. \label{morse bott orbit assumption 1}
    \item For any submanifold $Y\subseteq\crit(f)$, we denote its unstable manifold (resp. stable manifold) by $W^u(Y)$ (resp. by $W^s(Y)$). With these notations, we assume that $W^u(p)$ intersects $W^s(\mathcal{O})$ transversely for any critical point $p$ and any critical orbit $\mathcal{O}$. \label{morse bott orbit assumption 2}
\end{enumerate}
\end{assumption}

\begin{remark}
    \normalfont In this paper, we do not work with the $G$-equivariant cohomology like \cite[Section 4]{Austin1995} or \cite[Definition 7.1]{bgv}. Instead, we will carry out a different object: A $G$-invariant cohomology (cf. \cite[Definition 1.27]{felix2008algebraicmodels}) in Definition \ref{thom smale definition}. 
\end{remark}
  
    \begin{remark}\normalfont
        By \cite[Lemma 4.8]{wassermanequivariant}, the assumption \ref{morse bott orbit assumption 1} is generic. A recent progress \cite[Theorems 1.6 \& 1.8]{bao2026equivariantmorsebottcohomologystabilization} proves that the assumption \ref{morse bott orbit assumption 2} is generic. 
    \end{remark}

Since the critical orbits are isolated, and the symmetry brought by the group action reduces the curvature terms on the tubular neighborhood of $\crit(f)$ to a large extent, now it could be tempting to say: We could simply replace the Morse function in the original Witten instanton complex \cite{Bismut1994, Helffer01011985, wittendeformationweipingzhang} by our Morse-Bott function $f$, and then the instanton complex would be cochain isomorphic to the topologically constructed Thom-Smale complex of $f$ \cite[Section 3]{Austin1995}. However, the construction in \cite[Section 3]{Austin1995} uses the full space of differential forms on $\crit(f)$. Since  $\dim\crit(f)$ is usually $\geqslant 1$, then on the full space of differential forms on $\crit(f)$, the Hodge Laplacian on $\crit(f)$ still brings an unbounded series of eigenvalues. This triggers the obstacle again.  

 Fortunately, we observe that the space of $G$-invariant forms on each of these isolated critical orbits is finite dimensional. Thus, when restricted to the space of $G$-invariant forms on the critical set, the deformed Hodge Laplacian only has finitely many eigenvalues along the critical set. 
 
 Therefore in this paper, we need to work on both the topological side and the analytic side. Beginning with the topological side, we construct a much simplified $G$-invariant Thom-Smale complex associated with $f$ satisfying \ref{morse bott orbit assumption 1} and \ref{morse bott orbit assumption 2}. We replace the space of all differential forms on $\crit(f)$ in the original Thom-Smale complex \cite[Section 3]{Austin1995} by the space of $G$-invariant forms on $\crit(f)$. Then, we adjust the boundary map. After that, we show that this simplified complex computes the Betti numbers of $M$. Moreover, once we finish presenting the associated analytic side, i.e., the $G$-invariant version of the Witten instanton complex, we show that the analytic side is cochain isomorphic to the topological side.
 
We begin with the topological construction of the $G$-invariant Thom-Smale complex following \cite[Section 3]{Austin1995}. By \ref{morse bott orbit assumption 1}, if on the normal direction of a critical orbit $\mathcal{O}$, there are $i$ negative eigenvalues of the Hessian of $f$, then we say the Morse index of $\mathcal{O}$ is $i$. Let $\mathcal{O}_i$ be the union of all critical orbits $\mathcal{O}\subseteq\crit(f)$ with Morse index $ = i$. By \ref{morse bott orbit assumption 2}, the endpoint map
\begin{align}
    \pi_i: W^u(\mathcal{O}_i)\to\mathcal{O}_i
\end{align}
gives a fiber bundle over $\mathcal{O}_i$. We let $\mathcal{E}_i$ be the orientation bundle of the fiber bundle $W^u(\mathcal{O}_i)$. 

Let $\Omega^j(\mathcal{O}_i, \mathcal{E}_i)^G$ be the space of all $G$-invariant (See \cite[(1.1)]{bgv}) smooth $\mathcal{E}_i$-valued $j$-forms on $\mathcal{O}_i$. We define the space of $k$-cochains
\begin{align}
    C^k(M,f)^G\coloneqq \bigoplus_{i+j = k}\Omega^j(\mathcal{O}_i, \mathcal{E}_i)^G. 
\end{align}
Then, we let 
\begin{align}
    C^*(M,f)^G \coloneqq \bigoplus_{k = 0}^m C^k(M,f)^G.
\end{align}
 Since $\mathcal{O}_i$ is a union of isolated orbits, $C^k(M,f)^G$ is a finite-dimensional space. Furthermore, we obtain a boundary map 
\begin{align}\label{partial map mentioned in advance less technical}
    \partial: C^k(M,f)^G\to C^{k+1}(M,f)^G
\end{align}
based on the assumption \ref{morse bott orbit assumption 2}. 
Since \ref{morse bott orbit assumption 2} ensures the fiber bundle structure of the moduli spaces between critical orbits, the map $\partial$ is mainly given by integrating densities along the fibers of these moduli spaces. We give a detailed description of $\partial$ in (\ref{defn 1 of partial boundary map}) and (\ref{defn 2 of partial boundary map}). Then, we prove $\partial^2 = 0$ in Proposition \ref{boundary map verified as austin braam}.
\begin{definition}\label{thom smale definition}
\normalfont
    The complex given by $C^k(M,f)^G\ (0\leqslant k\leqslant m)$ and $\partial$ is called the $G$-invariant Thom-Smale complex of $M$ associated with $f$.
\end{definition}

Our first main result shows that the $G$-invariant Thom-Smale complex is well-defined and computes the Betti numbers of $M$. See Examples \ref{example 1 of topological result}-\ref{example 4 of topological result} for some computations.
\begin{theorem}\label{thom smale cohomology theorem}
Under the assumptions {\normalfont\ref{morse bott orbit assumption 1}} and {\normalfont\ref{morse bott orbit assumption 2}},
the $G$-invariant Thom-Smale complex of $M$ associated with $f$ is well-defined. It computes the de Rham cohomology of $M$.
\end{theorem}

\begin{remark}
    \normalfont The cohomology figured out by (\ref{partial map mentioned in advance less technical}) using $G$-invariant forms on $\crit(f)$ is not equal to the $G$-equivariant cohomology (See \cite[Section 4]{Austin1995} and \cite[Definition 7.1]{bgv}). As we will show in the following sections, instead of computing the $G$-equivariant cohomology, the cochain complex (\ref{partial map mentioned in advance less technical}) actually computes the $G$-invariant de Rham cohomology \cite[Definition 1.27]{felix2008algebraicmodels} of $M$, which is isomorphic to the classical de Rham cohomology of $M$ when $M$ is compact, and $G$ is compact and connected \cite[Theorem 1.28]{felix2008algebraicmodels}. 
\end{remark}

The $G$-invariant Thom-Smale complex provides a simplified way to compute the cohomology of $M$ using Morse-Bott functions whose critical sets have certain symmetry. The space of $G$-invariant forms is much smaller than the space of all smooth forms, bringing us a lot of convenience.

The main idea to prove Theorem \ref{thom smale cohomology theorem} is to use the spectral sequence method as in \cite[Section 3.3]{Austin1995}. In terms of technical details, we adapt necessary steps to fit our $G$-action situation.

Next, we present the analytic side, which is given by $G$-invariant eigenforms of the deformed Hodge Laplacian and admits a cochain isomorphism to the topological side. 

We let $\Omega^k(M)^G$ be the collection of all $G$-invariant smooth $k$-forms on $M$. For any $T>0$, we let 
\begin{align}
    d_T \coloneqq & e^{-Tf}de^{Tf}, \\
    D_{T} \coloneqq & d_T+d_T^*. 
\end{align}
Following \cite[(5.15)]{wittendeformationweipingzhang}, for any $\alpha>0$, we let 
\begin{align}
     F^k_T(M,f,\alpha)^G  
    \coloneqq\ & \bigoplus_{0\leqslant\lambda\leqslant\alpha}\left\{\omega\in\Omega^k(M)^G: D_T^2
\omega = \lambda\omega \right\},
\end{align}
where $d_T^*$ is the formal adjoint of $d_T$, and $D_T^2$ is the deformed Hodge Laplacian.
Similar to \cite[(5.15)]{wittendeformationweipingzhang}, we have $D_T^2d_T = d_TD_T^2$, making 
\begin{align}
    d_T: F_T^k(M,f,\alpha)^G\to F_T^{k+1}(M,f,\alpha)^G
\end{align}
a well defined cochain complex. We then let 
\begin{align}
    F^*_T(M,f,\alpha)^G \coloneqq \bigoplus_{k = 0}^m F^k_T(M,f,\alpha)^G.
\end{align}
    
\begin{definition}\label{witten instanton complex definition}
\normalfont
     The cochain complex given by $F^k_T(M,f,\alpha)^G$ $(0\leqslant k\leqslant m)$ and $d_{T}$ is called the $G$-invariant Witten instanton complex of $M$ associated with $f$. 
\end{definition}

According to the $G$-invariant version of Hodge theory (See Section \ref{section of Witten Laplacians and Hodge theory}), the $G$-invariant Witten instanton complex also computes the de Rham cohomology of $M$. 

As we see from \cite[(8.49)]{bismutandlebeau}, it is possible for us to analyze the operator $D_T$ and the Witten instanton complex associated with any $G$-invariant metric $\langle\cdot,\cdot\rangle$ under the assumptions \ref{morse bott orbit assumption 1} and \ref{morse bott orbit assumption 2}. However, to avoid the issues of non-isometric exponential maps (See \cite[(8.21), (8.57)]{bismutandlebeau}), as in \cite[(5.11)]{wittendeformationweipingzhang}, we adjust the $G$-invariant metric $\langle\cdot,\cdot\rangle$ on $M$ (See Section \ref{sections of metrics and connections} for the adjustment) according to the $G$-action and the $G$-equivariant Morse-Bott lemma \cite[Lemma 4.1]{wassermanequivariant} around each critical orbit. 

Moreover, by restricting the adjusted $\langle\cdot,\cdot\rangle$ to each critical orbit, we get a (twisted) Dirac type operator (See (\ref{definition of the twisted d plus d star rigorous definition here}) and Proposition \ref{twisted dirac type on critical orbits})
\begin{align}
    d+d^*: C^*(M,f)^G\to C^{*\pm 1}(M,f)^G.
\end{align}
An inner product on $C^*(M,f)^G$ is given by letting the basis induced by the $G$-action be orthonormal (See (\ref{inner product local expression on each critical orbit making d+d^* self adjoint}) for details). Under this inner product, $d+d^*$ on $C^*(M,f)^G$ is self-adjoint with a finite spectral radius. 
We let $\alpha_0$ be the spectral radius of the twisted $(d+d^*)^2$ on $C^*(M,f)^G$.  

When $T$ is sufficiently large, our second main result shows that Definition \ref{witten instanton complex definition} gives the analytic realization of the Thom-Smale complex associated with the $G$-invariant Morse-Bott function $f$ satisfying \ref{morse bott orbit assumption 1} and \ref{morse bott orbit assumption 2}. For some applications, see Corollaries \ref{application 1 of analytic result}-\ref{refinement of main result 2}. 
\begin{theorem}\label{isom given under a good metric be careful}
   We equip $M$ with the adjusted $\langle\cdot,\cdot\rangle$. For any $\alpha > \alpha_0$, 
        \begin{align}\label{expression of adjusted cochain isomorphism for section later}
    \Phi_{T}: F^k_T(M,f,\alpha)^G &\to C^k(M,f)^G \nonumber\\
    \omega&\mapsto\sum_{i = 0}^k(\pi_i)_*\left(\left.e^{Tf}\cdot\omega\right\vert_{\overline{W^u(\mathcal{O}_i)}}\right) \ \ (k = 0,1,\cdots,m)
\end{align}
is a cochain isomorphism when $T$ is sufficiently large. Thus, $F^k_T(M,f,\alpha)^G$ and $d_T$ constructs the analytic side. 
\end{theorem}

\begin{remark}\normalfont
    By \cite[Theorem 7.10]{bott2013differential}, the map $(\pi_i)_*$ integrates forms along fibers to get $\mathcal{E}_i$-valued forms.
\end{remark}

    The most interesting two things in Theorem \ref{isom given under a good metric be careful} are the choice of constant $\alpha_0$ and the condition $\alpha>\alpha_0$, revealing how we separate eigenvalues so that the analytic machinery corresponds to the topological one. First, they show that on the cochain complex level, the nonzero eigenvalues coming from the horizontal direction around $\crit(f)$ has a nontrivial effect, which may be underestimated if we only look at the Morse-Bott inequalities. Second, they confirm that our usage of $G$-invariance truly reduces the number of eigenvalues on the horizontal direction and makes the analysis required in the proof of the cochain isomorphism (\ref{expression of adjusted cochain isomorphism for section later}) work well.

\begin{remark}\normalfont
    When $G$ is a torus, we have $\alpha_0 = 0$. In particular, when $G$ is a torus, if $\mathcal{O}_i$ is a critical orbit admitting nonorientable $W^u(\mathcal{O}_i)$, then $\Omega^j(\mathcal{O}_i,\mathcal{E}_i) = 0$.
     Actually, by Examples \ref{TX and MN} and \ref{irrelevant to group actions}, the constant $\alpha_0$ relies on both $M$ and $G$. 
\end{remark}

Following \cite{bismutandlebeau}, the main idea to prove Theorem \ref{isom given under a good metric be careful} is to write $D_T$ into
\begin{align}\label{main idea separation expression}
    D^H+D^V+T\hat{c}(df) + \text{tail terms}
\end{align}
around $\crit(f)$. 
Here, we identify $(D^H)^2$ with the Hodge Laplacian on $\crit(f)$, which contributes the finite spectral radius $\alpha_0$. The $(D^V+T\hat{c}(df))^2$ is the harmonic oscillator along the normal direction of $\crit(f)$, and its non-zero eigenvalues are $\gg \alpha_0$ when $T\to+\infty$.

It is worthwhile to mention that the curvature terms in the tail of (\ref{main idea separation expression}) are greatly simplified by the $G$-action. Once the metric is adjusted based on the $G$-equivariant Morse lemma \cite[Lemma 4.1]{wassermanequivariant}, then $D^H$ and $D^V+ T\hat{c}(df)$ anti-commute with each other (See Proposition \ref{commutator} and cf. \cite[(9.69)]{bismutandlebeau}). Thus, we can treat $(D^H)^2$ and $(D^V+T\hat{c}(df))^2$ separately without facing their intertwining terms. Also, when restricted to each fiber of the normal bundle of $\crit(f)$, the tail terms of (\ref{main idea separation expression}) consist of only first order terms. Thus, we carry out a simplified version of \cite[(10.15)]{bismutandlebeau}.

This paper is organized in the following order. First, in Sections \ref{section of Invariant Thom-Smale complexes} and \ref{section of Quasi-isomorphism in topology}, we clarify Definition \ref{thom smale definition} and prove Theorem \ref{thom smale cohomology theorem}. Then, in Sections \ref{section of Witten Laplacians and Hodge theory} and \ref{sections of metrics and connections}, we clarify Definition \ref{witten instanton complex definition} and present prerequisites that are necessary for the subsequent analysis. Afterwards, in Sections \ref{section of spectral gaps} and \ref{section of cochain isomorphisms}, by asymptotic analysis on $D_T$, we prove Theorem \ref{isom given under a good metric be careful}. Finally, in Section \ref{section of computations and examples}, we give some examples and corollaries of Theorem \ref{thom smale cohomology theorem} and Theorem \ref{isom given under a good metric be careful}.

\ 

\noindent\textbf{Acknowledgments.}
This paper presents one of the projects in my PhD thesis. I want to sincerely thank my PhD supervisor \professor Xiang Tang for invaluable guidance, many inspiring discussions, and helpful suggestions on writing. 
Meanwhile, I want to thank \professor Erkao Bao, \professor Aliakbar Daemi, \professor Xianzhe Dai, \professor Rui Loja Fernandes, \professor Nigel Higson, \professor Tian-Jun Li, \professor Yanli Song, \professor Li-Sheng Tseng, \professor Jinmin Wang, \professor Zhizhang Xie, \professor Guoliang Yu, \doctor Qiaochu Ma, \doctor Shengzhen Ning, and \doctor Jesus Sanchez for their practical comments, and \doctor Haley Dolosic for her lectures on academic writing. 
Moreover, I want to express my appreciation to the McDonnell International Scholars Academy at WUSTL for the McDonnell Scholarship and the Lo Fellowship supporting my PhD study. 

\section{G-invariant construction}\label{section of Invariant Thom-Smale complexes}
In this section, we explain the construction of the $G$-invariant Thom-Smale complex and complete the first half of Theorem \ref{thom smale cohomology theorem}. To clarify our convention, we do not adjust the Riemannian metric $\langle\cdot,\cdot\rangle$ at present. We will only use the adjusted one in Sections \ref{sections of metrics and connections}, \ref{section of spectral gaps}, and \ref{section of cochain isomorphisms}. 

As in \cite[(1.1)]{bgv}, 
we have a left $G$-action on $\mathcal{E}_i$-valued $j$-forms. It allows us to define $\Omega^j(\mathcal{O}_i, \mathcal{E}_i)^G$, the space of all $G$-invariant $\mathcal{E}_i$-valued smooth $j$-forms. Since $\mathcal{E}_i$ is a flat bundle, the de Rham $d$ is still well-defined on  $\Omega^j(\mathcal{O}_i, \mathcal{E}_i)^G$ (See \cite[\S7]{bott2013differential}). 

Let $\mathbb{B}^i\subseteq\mathbb{R}^i$ be the open unit ball of dimension $i$. Then, given an open cover $\{U_a\}$ of $\mathcal{O}_i$ and the associated local trivializations
$\varphi_a: U_a\times\mathbb{B}^i\to W^u(\mathcal{O}_i)$
of the fiber bundle $\pi_i: W^u(\mathcal{O}_i)\to\mathcal{O}_i$, we have transition maps
\begin{align}
    \varphi_{ab}: U_a\cap U_b &\to \text{Diff}(\mathbb{B}^i) \nonumber\\
    p&\mapsto \varphi^{-1}_a(p)\circ\varphi_b(p).
\end{align}
Let $\sgn$ be the function on $\text{Diff}(\mathbb{B}^i)$ mapping orientation-preserving (resp. orientation-reversing) diffeomorphisms to $1$ (resp. $-1$).
The following transitions
\begin{align}
    \sgn\circ\varphi_{ab}: U_a\cap U_b\to\{\pm 1\}
\end{align}
define the orientation bundle $\mathcal{E}_i$. This $\mathcal{E}_i$ admits a $G$-action given as follows. For any $g\in G$, and any $(p,t)$ in a local trivialization $U_a\times\mathbb{R}$ of $\mathcal{E}_i$, we find another $U_b\times\mathbb{R}$ such that $gp\in U_b$. Then, $g(p,t)$ is defined to be 
\begin{align}
    \left(gp, t\cdot\sgn(\varphi_b^{-1}(gp)\circ g\circ\varphi_a(p))\right)\in U_b\times\mathbb{R}
\end{align}
if written in the local trivialization $U_b\times\mathbb{R}$. 

For any $r\in\mathbb{Z}_{\geqslant 0}$, we have a moduli space
\begin{align}
    \mathcal{M}(\mathcal{O}_{i+r}, \mathcal{O}_{i})\coloneqq \left(W^u(\mathcal{O}_{i+r})\cap W^s(\mathcal{O}_{i})\right)/\mathbb{R}
\end{align}
consisting of the flow lines of $-\nabla f$ from $\mathcal{O}_{i+r}$ to $\mathcal{O}_{i}$. This ``quotient $\mathbb{R}$'' means quotient the time parameter of each flow line. 
Then, we have the following two natural maps 
\begin{align}
    \ell_{i}^{i+r}: \mathcal{M}(\mathcal{O}_{i+r}, \mathcal{O}_{i})\to \mathcal{O}_{i}\ \ \text{and}\ \ u_{i}^{i+r}: \mathcal{M}(\mathcal{O}_{i+r}, \mathcal{O}_{i})\to\mathcal{O}_{i+r}.
\end{align}
By \ref{morse bott orbit assumption 2}, $u_{i}^{i+r}$ makes $\mathcal{M}(\mathcal{O}_{i+r}, \mathcal{O}_{i})$ a fiber bundle over $\mathcal{O}_{i+r}$.

Similar to \cite[Section Vl.4.c]{micheleaudin}, for any $\omega\in\Omega^j(\mathcal{O}_i, \mathcal{E}_i)^G$, we construct a form 
\begin{align}
    (u_i^{i+r})_*(\ell_i^{i+r})^*\omega\in\Omega^{j-r+1}(\mathcal{O}_{i+r},\mathcal{E}_{i+r})^G, 
\end{align}
where $1\leqslant r\leqslant j+1$. In fact, $(\ell_i^{i+r})^*\omega\in\Omega^j\left(\mathcal{M}(\mathcal{O}_{i+r},\mathcal{O}_i),(\ell_i^{i+r})^*\mathcal{E}_i\right)$ is given by
\begin{align}
    & \left((\ell_i^{i+r})^*\omega\right)_\gamma(\nu_1,\cdots,\nu_j) \nonumber\\
    =\ & \omega_{\ell^{i+r}_i(\gamma)}\left((d\ell_i^{i+r})\nu_1,\cdots, (d\ell_i^{i+r})\nu_j\right)\in \text{the\ fiber\ of\ }\mathcal{E}_i\text{\ at\ }\ell_i^{i+r}(\gamma)
\end{align}
for any $\gamma\in\mathcal{M}(\mathcal{O}_{i+r},\mathcal{O}_i)$ and any $\nu_1,\cdots,\nu_j\in T_\gamma\mathcal{M}(\mathcal{O}_{i+r},\mathcal{O}_i)$.

In addition, for any $q\in\mathcal{O}_{i+r}$ and $\gamma\in\mathcal{M}(q,\mathcal{O}_{i})$, there is the following isomorphism 
\begin{align}\label{hutchings notes isomorphism}
    T_qW^u(q)\cong T_\gamma\mathcal{M}(q,\mathcal{O}_i)\oplus T\gamma\oplus T_{\ell^{i+r}_i(\gamma)}W^u(\ell^{i+r}_i(\gamma))
\end{align}
according to
\cite[Section 6.2.3]{hutchingslecturenotes}. Thus, given an orientation on $W^u(q)$ and $j-r+1$ tangent vectors $w_1,\cdots,w_{j-r+1}\in T_q\mathcal{O}_{i+r}$, we get a density $\xi$ on $\mathcal{M}(q, \mathcal{O}_i)$ in this way: 
For any $\gamma\in\mathcal{M}(q, \mathcal{O}_i)$ and any $z_1,\cdots,z_{r-1}\in T_\gamma\mathcal{M}(q,\mathcal{O}_i)$, $\xi$ at $\gamma$ is given by  
\begin{align}\label{alpha technical}
\begin{split}
    &\xi_\gamma(z_1,\cdots,z_{r-1})\\
    =\  &\left((\ell^{i+r}_i)^*\omega\right)_\gamma\left(\left.\frac{d}{dt}\right|_{t = 0}\exp(te_1)\gamma, \cdots, \left.\frac{d}{dt}\right|_{t = 0}\exp(te_{j-r+1})\gamma,z_1,\cdots,z_{r-1}\right)\\
    \in\ & \mathcal{E}_i\cong\text{the\ orientation\ bundle\ of\ }\mathcal{M}(q,\mathcal{O}_i)\ \ \text{(the isomorphism is by (\ref{hutchings notes isomorphism}))},
    \end{split}
\end{align}
where $e_1,\cdots,e_{j-r+1}$ are any vectors in $\mathfrak{g}$ satisfying
\begin{align}\label{write vectors into lie algebra elements}
    w_1 = \left.\frac{d}{dt}\right|_{t = 0}\exp(te_1)q,\ \cdots,\ w_{j-r+1} = \left.\frac{d}{dt}\right|_{t = 0}\exp(te_{j-r+1})q.
\end{align}
The definition of $\xi$ is independent of the choice of $e_1,\cdots,e_{j-r+1}$. The integration of $\xi$ on $\mathcal{M}(q,\mathcal{O}_i)$ gives us a real number. However, if we choose another orientation on $W^u(q)$, we get the opposite value. With an abuse of notation, we say that
\begin{align}
    \int_{\mathcal{M}(q,\mathcal{O}_i)}\xi\in \mathcal{E}_{i+r}.
\end{align}
Then, $(u_i^{i+r})_*(\ell_i^{i+r})^*\omega$ at each point $q\in\mathcal{O}_{i+r}$ is defined by
\begin{align}
    \left((u_i^{i+r})_*(\ell_i^{i+r})^*\omega\right)_q(w_1,\cdots,w_{j-r+1}) = \int_{\mathcal{M}(q,\mathcal{O}_i)}\xi
\end{align}
for any $w_1,\cdots,w_{j-r+1}\in T_q\mathcal{O}_{i+r}$. 
\begin{lemma}
    This $(u_i^{i+r})_*(\ell_i^{i+r})^*\omega$ is $G$-invariant. 
\end{lemma}
\begin{proof}
    For $q\in\mathcal{O}_{i+r}$, we let 
    \begin{align}
        \varphi_a: U_a\times\mathbb{B}^{i+r}\to W^u(\mathcal{O}_{i+r})
    \end{align}
     be the local trivialization around $q$. Then, for any $g\in G$, we let 
     \begin{align}
         g^{-1}\circ\varphi_a\circ g: (g^{-1}U_a)\times\mathbb{B}^{i+r}\to W^u(\mathcal{O}_{i+r})
     \end{align}
     be the local trivialization around $g^{-1}q$. Under these two trivializations and their induced ones on $\mathcal{E}_{i+r}$, we find that
     \begin{align}
         \sgn\left(\varphi_a^{-1}(q)\circ g\circ (g^{-1}\circ\varphi_a\circ g)(g^{-1}q)\right) = 1.
     \end{align}
    Thus, 
    for any $w_1,\cdots,w_{j-r+1}\in T_q\mathcal{O}_{i+r}$, 
    \begin{align}
    & \left(g\cdot (u_i^{i+r})_*(\ell_i^{i+r})^*\omega\right)_q(w_1,\cdots,w_{j-r+1}) \nonumber\\
    =\ & \left((u_i^{i+r})_*(\ell_i^{i+r})^*\omega\right)_{g^{-1}q}(g^{-1}_{*}w_1,\cdots,g^{-1}_{*}w_{j-r+1}).
    \end{align}
    With the same notations $e_1,\cdots, e_{j-r+1}$ as (\ref{write vectors into lie algebra elements}), for any $\tau\in \mathcal{M}(g^{-1}q, \mathcal{O}_i)$ and $\nu_1,\cdots,\nu_{r-1}\in T_\tau\mathcal{M}(g^{-1}q, \mathcal{O}_i)$, \text{since\ $\omega$\ is\ $G$-invariant}, we define a density $\beta$ on $\mathcal{M}(g^{-1}q,\mathcal{O}_i)$ and find
    {\small\begin{align}
           &\beta_\tau(\nu_1,\cdots,\nu_r)  \nonumber\\
        = &\left((\ell^{i+r}_i)^*\omega\right)_\tau\left(\left.\frac{d}{dt}\right|_{t = 0}\exp(t\adjoint_{g^{-1}}e_1)\tau, \cdots, \left.\frac{d}{dt}\right|_{t = 0}\exp(t\adjoint_{g^{-1}}e_{j-r+1})\tau,\nu_1,\cdots,\nu_{r-1}\right)  \nonumber\\
         %= &\left((l^{i+r}_i)^*\omega\right)_{g^{-1}g\tau}\left(\left.\frac{d}{dt}\right|_{t = 0}g^{-1}\exp(te_1)g\tau, \cdots, \left.\frac{d}{dt}\right|_{t = 0}g^{-1}\exp(te_{j-r+1})g\tau,g^{-1}_*g_*\nu_1,\cdots,g^{-1}_*g_*\nu_{r-1}\right)\\
         =&\ \omega_{g^{-1}\ell^{i+r}_i(g\tau)}\left(\left.\frac{d}{dt}\right|_{t = 0}g^{-1}\exp(te_1)\ell^{i+r}_i(g\tau), \cdots, \left.\frac{d}{dt}\right|_{t = 0}g^{-1}\exp(te_{j-r+1})\ell^{i+r}_i(g\tau)\right.,  \nonumber\\
         & \qquad \qquad \qquad\qquad\qquad\qquad \qquad\qquad\qquad\qquad g^{-1}_*(d\ell^{i+r}_i)g_*\nu_1,\cdots,g^{-1}_*(d\ell^{i+r}_i)g_*\nu_{r-1}\biggr)  \nonumber\\
         = & \ g^{-1}\cdot\hspace{-1mm}\biggl[\omega_{\ell^{i+r}_i(g\tau)}\left(\left.\frac{d}{dt}\right|_{t = 0}\exp(te_1)\ell^{i+r}_i(g\tau), \cdots, \left.\frac{d}{dt}\right|_{t = 0}\exp(te_{j-r+1})\ell^{i+r}_i(g\tau),\right.(d\ell^{i+r}_i)g_*\nu_1,\cdots,(d\ell^{i+r}_i)g_*\nu_{r-1}\biggr)\hspace{-0.5mm}\biggr]  \nonumber\\
         = &\ g^{-1}\cdot \xi_{g\tau}(g_*\nu_1,\cdots,g_*\nu_{r-1}).
    \end{align}}
    According to \cite[Proposition 16.41]{lee2012introduction}, 
    \begin{align}
        & \left((u_i^{i+r})_*(\ell_i^{i+r})^*\omega\right)_{g^{-1}q}(g^{-1}_{*}w_1,\cdots,g^{-1}_{*}w_{j-r+1})   \nonumber\\
        =\ & \int_{\mathcal{M}(g^{-1}q,\mathcal{O}_i)}\beta   \nonumber\\
        =\ & \int_{\mathcal{M}(q,\mathcal{O}_i)}\xi,
    \end{align}
    which is exactly $\left((u_i^{i+r})_*(\ell_i^{i+r})^*\omega\right)_q(w_1,\cdots,w_{j-r+1})$. 
\end{proof}

We now precisely define $\partial$ in (\ref{partial map mentioned in advance less technical}). For each $\Omega^j(\mathcal{O}_i, \mathcal{E}_i)^G$ and $r = 0,1,\cdots, j+1$, we define 
\begin{align}\label{defn 1 of partial boundary map}
\begin{split}
    \partial_r: \Omega^j(\mathcal{O}_i, \mathcal{E}_i)^G &\to\Omega^{j-r+1}(\mathcal{O}_{i+r}, \mathcal{E}_{i+r})^G\\
    \omega &\mapsto \begin{cases} 
      d\omega & \text{if}\ r = 0,\\
      (-1)^j\left(u_{i}^{i+r}\right)_*\left(\ell_{i}^{i+r}\right)^*\omega & \text{if\ } 1\leqslant r\leqslant j+1.
   \end{cases}
   \end{split}
\end{align}
Here, $\left(u^{i+r}_{i}\right)_*$ is given by integrating densities along fibers. Then, we define  
\begin{align}\label{defn 2 of partial boundary map}
    \partial: C^k(M,f)^G & \to C^{k+1}(M,f)^G  \nonumber\\
    \Omega^{j}(\mathcal{O}_i, \mathcal{E}_i)^G \ni\omega &\mapsto\partial_0\omega+\partial_1\omega+\cdots+\partial_{j+1}\omega.
\end{align}
To verify that $\partial^2 = 0$, we need to clarify the manifold with corner structure of $\overline{\mathcal{M}(\mathcal{O}_{i+r},\mathcal{O}_i)}$, the local orientation of the codimension-$1$ stratum of $\overline{\mathcal{M}(\mathcal{O}_{i+r},\mathcal{O}_i)}$, and a  generalized Stokes' theorem for the integration on 
\begin{align}
    u_i^{i+r}: \overline{\mathcal{M}(\mathcal{O}_{i+r},\mathcal{O}_i)}\to \mathcal{O}_{i+r}
\end{align}
along fibers. 

First, following the notations in \cite[Section 3.2]{Austin1995}, for integers $0<r_1<\cdots<r_a<r$, we let
\begin{align}
    & X_{r_1\cdots r_a} \nonumber\\
    \coloneqq\ &  \mathcal{M}(\mathcal{O}_{i+r},\mathcal{O}_{i+r_a})\times_{\mathcal{O}_{i+r_a}}\cdots\times_{\mathcal{O}_{i+r_2}}\mathcal{M}(\mathcal{O}_{i+r_2},\mathcal{O}_{i+r_1})\times_{\mathcal{O}_{i+r_1}}\mathcal{M}(\mathcal{O}_{i+r_1},\mathcal{O}_{i}).
\end{align}
According to \cite[Lemma 2.6, Lemma 3.3, Theorem A.11]{Austin1995}, $\overline{\mathcal{M}(\mathcal{O}_{i+r},\mathcal{O}_i)}$ is a smooth manifold with corner:
\begin{theorem}[Austin-Braam \cite{Austin1995}, 1995]\label{austin braam asymptotic gradient} The compactification $\overline{\mathcal{M}(\mathcal{O}_{i+r},\mathcal{O}_i)}$ is a smooth manifold with corner: 

{\normalfont (i)} The codimension-$1$ stratum of $\overline{\mathcal{M}(\mathcal{O}_{i+r},\mathcal{O}_i)}$ is equal to $\bigcup_{0<s<r} X_s.$ More precisely, there is an embedding
\begin{align}\label{before restricting psi_s to zero}
    \Psi_{s}: (-1,0]\times X_{s}\to \overline{\calm(\mathcal{O}_{i+r},\calo_i)}
    %\underbrace{(-1,0]\times\cdots\times (-1,0]}_{a \text{\ \normalfont copies}}\times X_{r_1r_2\cdots r_a} \to \overline{\calm(\mathcal{O}_{i+r},\calo_i)}
    \end{align}
mapping $(-1,0]\times X_{s}$ to an open subset of $\overline{\mathcal{M}(\mathcal{O}_{i+r},\mathcal{O}_i)}$. In particular, $\{0\}\times X_{s}$ is mapped into the codimension-$1$ stratum of $\overline{\mathcal{M}(\mathcal{O}_{i+r},\mathcal{O}_i)}$. 

{\normalfont (ii)} Moreover, $\Psi_s$ induces an embedding
\begin{align}
    \Psi_{r_1\cdots r_a}: 
    \underbrace{(-1,0]\times\cdots\times (-1,0]}_{a \text{\ \normalfont copies}}\times X_{r_1\cdots r_a} \to \overline{\calm(\mathcal{O}_{i+r},\calo_i)}.
    \end{align}
     The codimension-$a$ stratum of $\overline{\mathcal{M}(\mathcal{O}_{i+r},\mathcal{O}_i)}$ is $\bigcup_{0<r_1<\cdots<r_a<r}X_{r_1\cdots r_a}$. 
\end{theorem}

Second, we describe the local orientation of the codimension-$1$ stratum of $\mathcal{M}(\mathcal{O}_{i+r},\mathcal{O}_i)$ induced by (\ref{before restricting psi_s to zero}). Using the $G$-action, for any $q\in\mathcal{O}_{i+r}$, we can find an open subset $\mathcal{U}_{i+s}$ (resp. $\mathcal{U}_i$) of $\mathcal{O}_{i+s}$ (resp. $\mathcal{O}_i$) and some $q'\in\mathcal{U}_{i+s}$ (resp. $q''\in\mathcal{U}_i$) such that $W^u(\mathcal{U}_{i+s})$ (resp. $W^u(\mathcal{U}_i)$) is diffeomorphic to $\mathcal{U}_{i+s}\times W^u(q')$ (resp. $\mathcal{U}_{i}\times W^u(q'')$), and the fiber bundle 
\begin{align}
    \mathcal{M}(q,\mathcal{U}_{i+s})\times_{\mathcal{U}_{i+s}}\mathcal{M}(\mathcal{U}_{i+s},\mathcal{O}_i)\to\mathcal{M}(q,\mathcal{U}_{i+s})
\end{align}
over 
$\mathcal{M}(q,\mathcal{U}_{i+s})$
is trivialized as the fiber bundle 
\begin{align}
    \mathcal{M}(q,\mathcal{U}_{i+s})\times\mathcal{M}(q',\mathcal{O}_i)\to\mathcal{M}(q,\mathcal{U}_{i+s})
\end{align}
with fiber $\mathcal{M}(q',\mathcal{O}_i)$. By the isomorphism (\ref{hutchings notes isomorphism}), once we select the orientations of unstable manifolds $W^u(q)$, $W^u(q')$, and $W^u(q'')$, we automatically obtain the orientations of $\mathcal{M}(q,\mathcal{U}_i)$, $\mathcal{M}(q,\mathcal{U}_{i+s})$, and $\mathcal{M}(q',\mathcal{U}_i)$. We denote these three orientations by $[\mathcal{M}(q,\mathcal{U}_i)]$, 
 $[\mathcal{M}(q,\mathcal{U}_{i+s})]$, and $[\mathcal{M}(q',\mathcal{U}_i)]$ respectively. Then, we have the following Lemma \cite[Lemma 3.4]{Austin1995}. 

\begin{lemma}[Austin-Braam \cite{Austin1995}, 1995]\label{theorem austin braam orientation induced on the boundary adjusted}
    The orientation of the manifold 
    \begin{align}
        \mathcal{M}(q,\mathcal{U}_{i+s})\times_{\mathcal{U}_{i+s}}\mathcal{M}(\mathcal{U}_{i+s},\mathcal{U}_i)
    \end{align}
    induced by $[\mathcal{M}(q,\mathcal{U}_i)]$ under {\normalfont(\ref{before restricting psi_s to zero})} is equal to 
    \begin{align}
        (-1)^{r-s-1}[\mathcal{M}(q,\mathcal{U}_{i+s})][\mathcal{M}(q',\mathcal{U}_i)].
    \end{align} 
\end{lemma}

Third, we give the generalized Stokes' theorem. Continuing with the results and the notations in Theorem \ref{austin braam asymptotic gradient}, for any $q\in\mathcal{O}_{i+r}$, we let
\begin{align}
    X_{s}(q) \coloneqq \mathcal{M}(q,\mathcal{O}_{i+s})\times_{\mathcal{O}_{i+s}}\mathcal{M}(\mathcal{O}_{i+s},\mathcal{O}_{i}),
\end{align}
let $\Psi_s' = \Psi_s\vert_{\{0\}\times X_s(q)}$, and let $\text{\normalfont Or}(q,\mathcal{O}_i)$ be the orientation bundle of $\overline{\mathcal{M}(q,\mathcal{O}_i)}$. We 
state the generalized Stokes' theorem on each fiber of 
\begin{align}
    u_i^{i+r}: \mathcal{M}(\mathcal{O}_{i+r},\mathcal{O}_i)\to\mathcal{O}_i.
\end{align}
\begin{lemma}\label{stokes}
For any $\omega\in \Omega^{r-2}\left(\overline{\mathcal{M}(q,\mathcal{O}_i)},\text{\normalfont Or}(q,\mathcal{O}_i)\right)$, we have 
\begin{align}
    \int_{\overline{\mathcal{M}(q,\mathcal{O}_{i})}} d\omega = \sum_{0<s<r}\int_{\overline{X_s(q)}} (\Psi_s')^*\omega.
\end{align}
\end{lemma}
\begin{proof}
 Let $t$ be the coordinate on $(-1,0]$, $U = (x_1,\cdots,x_{\delta_i})$ be a local coordinate of $X_{s}(q)$, and $e$ be the section $1$ of $\text{\normalfont Or}(q,\mathcal{O}_i)\vert_{\Psi_s((-1,0]\times U)}$ induced by $dt\wedge dx_1\wedge\cdots\wedge dx_{\delta_i}$. Meanwhile, we recall that $\dim\mathcal{M}(q,\mathcal{O}_i) = r-1$. 
 
 Suppose that $\supp(\omega)\subseteq \Psi_s((-1,0]\times U)$, we write
 \begin{align}
     \omega = \varphi_0 dx_1\wedge\cdots\wedge dx_{\delta_i}\otimes e + \sum_{j = 1}^{r-2}\varphi_{j}dt\wedge dx_1\wedge\cdots\wedge\widehat{dx_j}\wedge\cdots\wedge dx_{r-2}\otimes e
 \end{align}
 and then find
 \begin{align}
    \int_{\overline{\mathcal{M}(q,\mathcal{O}_i)}}d\omega =\ &\int_{(-1,0]\times U} d(\varphi_0 dx_1\wedge\cdots\wedge dx_{r-2}\otimes e)   \nonumber\\
    & + \sum_{j = 1}^{r-2}\int_{(-1,0]\times U}d(\varphi_{j}dt\wedge dx_1\wedge\cdots\wedge\widehat{dx_j}\wedge\cdots\wedge dx_{r-2}\otimes e)   \nonumber\\
     =\ & \int_{(-1,0]\times U} \dfrac{\partial\varphi_0}{\partial t} |dtdx_1\cdots dx_{r-2}| + \sum_{j = 1}^{r-2}\int_{(-1,0]\times U} (-1)^{j}\dfrac{\partial \varphi_{j}}{\partial x_j}|dtdx_1\cdots dx_{r-2}|   \nonumber\\
     =\ & \int_U \varphi_0(0,x_1,\cdots,x_{r-2})|dx_1\cdots dx_{r-2}|    \nonumber\\
     =\ & \int_{\overline{X_s(q)}}(\Psi_s')^*\omega.
 \end{align}
The same proof applies to the situation where $\omega$ is supported in other types of local charts given by $\Psi_{r_1\cdots r_a}$. The general case is obtained by partition of unity.   
\end{proof}
\begin{remark}\normalfont
    Similar to \cite[Theorem 16.48]{lee2012introduction}, we need a nonvanishing outward-pointing direction induced by  $(-1,0]$ along $X_s(q)$. This is to ensure that $(\Psi'_s)^*\omega$ is a density. 
\end{remark}

Notice the fiber bundle structure 
\begin{align}
    u_{i,\partial}^{i+r}: \bigcup_{0<s<r}\mathcal{M}(\mathcal{O}_{i+r},\mathcal{O}_{i+s})\times_{\mathcal{O}_{i+s}}\mathcal{M}(\mathcal{O}_{i+s},\mathcal{O}_i)\to\mathcal{O}_{i+r}
\end{align}
induced by $u_i^{i+r}$ on the codimension-1 stratum. We let 
\begin{align}
    \Psi': \bigcup_{0<s<r}\mathcal{M}(\mathcal{O}_{i+r},\mathcal{O}_{i+s})\times_{\mathcal{O}_{i+s}}\mathcal{M}(\mathcal{O}_{i+s},\mathcal{O}_i)\to\overline{\mathcal{M}(\mathcal{O}_{i+r},\mathcal{O}_i)}
\end{align}
be the map obtained by restricting each $\Psi_s$ ($0<s<r$) in (\ref{before restricting psi_s to zero}) to $\{0\}\times X_s$. We immediately have the following corollary of Lemma \ref{stokes} (compare with \cite[(3.3)]{Austin1995}).  
\begin{corollary}\label{corollary of the above stokes}
For any $\omega\in\Omega^j(\calo_i,\mathcal{E}_i)^G$, we have the formula
   \begin{align}
       & (u_i^{i+r})_*\left(d(\ell^{i+r}_i)^*\omega)\right) \nonumber\\
        =\ & d(u_i^{i+r})_*(\ell^{i+r}_i)^*\omega + (-1)^{j-r+2} (u^{i+r}_{i,\partial})_*(\Psi')^*(\ell^{i+r}_i)^*\omega. 
    \end{align}
\end{corollary}
\begin{proof}
    Around $q\in\mathcal{O}_{i+r}$, by \cite[Theorem 9.3.7(iii)]{2011structureandgeometryofliegroups}, we find $e_1,\cdots,e_n\in\mathfrak{g}$ such that there is a small $\delta > 0$ making
     \begin{align}
        \mathbb{B}^n(\delta)\coloneqq\{(x_1,\cdots,x_n)\in\mathbb{R}^n: x_1^2+\cdots+x_n^2<\delta\}&\to\mathcal{O}_{i+r}   \nonumber\\
        (x_1,\cdots,x_n)&\mapsto \exp(x_1e_1+\cdots+x_ne_n)q
     \end{align}
     an embedding. Then, we obtain a local trivialization 
     \begin{align}\label{local trivialization of fiber bundle of moduli}
         \mathbb{B}^n(\delta)\times\mathcal{M}(q,\mathcal{O}_i)&\to\mathcal{M}(\mathcal{O}_{i+r},\mathcal{O}_i)  \nonumber\\
         (x_1,\cdots,x_n,\gamma)&\mapsto \exp(x_1e_1+\cdots+x_ne_n)\gamma 
     \end{align}
     of 
    $u_i^{i+r}: \mathcal{M}(\mathcal{O}_{i+r},\mathcal{O}_i)\to\mathcal{O}_{i+r}$. On the local trivialization (\ref{local trivialization of fiber bundle of moduli}), we write 
    \begin{align}
        (\ell_i^{i+r})^*\omega = \sum_{a = 0}^n\sum_{1\leqslant k_1<\cdots<k_{a}\leqslant n}\varphi_{k_1\cdots k_a} dx_{k_1}\wedge\cdots\wedge dx_{k_{a}}\wedge\eta_{k_1\cdots k_a},
    \end{align}
    where $\varphi_{k_1\cdots k_a}$ is a smooth function, and $\eta$ is an $\text{Or}(q,\mathcal{O}_i)$-valued form on $\mathcal{M}(q,\mathcal{O}_i)$ with degree $j-a$. 
    Recall that $r-1 = \dim\mathcal{M}(q,\mathcal{O}_i)$, and let $d_q$ be the de Rham exterior derivative on $\mathcal{M}(q,\mathcal{O}_i)$. Then, we find 
    \begin{align}
        & (u_i^{i+r})_*\left(d(\ell^{i+r}_i)^*\omega)\right)    \nonumber\\
        =\ & (u_i^{i+r})_*d\left(\sum_{a = 0}^n\sum_{1\leqslant k_1<\cdots<k_{a}\leqslant n}\varphi_{k_1\cdots k_{j-r+2}} dx_{k_1}\wedge\cdots\wedge dx_{k_{a}}\wedge\eta_{k_1\cdots k_{a}}\right)    \nonumber\\
        =\ & (u_i^{i+r})_*\left(\sum_{\substack{1\leqslant k_1<\cdots<k_{j-r+1}\leqslant n\\\alpha\notin\{k_1,\cdots,k_{j-r+1}\}}}\dfrac{\partial}{\partial x_\alpha}\varphi_{k_1\cdots k_{j-r+1}}dx_\alpha\wedge dx_{k_1}\wedge\cdots\wedge dx_{k_{j-r+1}}\wedge\eta_{k_1\cdots k_{j-r+1}}\right)   \nonumber\\
        & + (-1)^{j-r+2} (u_i^{i+r})_*\left(\sum_{1\leqslant k_1<\cdots<k_{a}\leqslant n}dx_{k_1}\wedge\cdots\wedge dx_{k_{j-r+2}}\wedge d_q(\varphi_{k_1\cdots k_{j-r+2}}\eta_{k_1\cdots k_{j-r+2}})\right)   \nonumber\\
        =\ & d(u_i^{i+r})_*(\ell^{i+r}_i)^*\omega + (-1)^{j-r+2} (u^{i+r}_{i,\partial})_*(\Psi')^*(\ell^{i+r}_i)^*\omega.
    \end{align}
    The second term in the last row is by Lemma \ref{stokes}. 
\end{proof}

\begin{remark}\normalfont
    The definition of $(u_{i,\partial}^{i+r})_*$ follows the same pattern as defining $(u_i^{i+r})_*$. Given any $q\in\calo_{i+r}$, once we choose an orientation of $W^u(q)$ and any  $w_1,\cdots,w_{j-r+1}\in T_q\mathcal{O}_{i+r}$, similar to (\ref{alpha technical}), $(\Psi')^*(\ell^{i+r}_i)^*\omega$ induces a density $(\Psi')^*\xi$ on the boundary of $\calm(q,\calo_i)$. Then, applying Lemma \ref{stokes}, we integrate the density $(\Psi')^*\xi$ on the codimension-$1$ stratum 
    \begin{align}
\bigcup_{0<s<r}\left(\calm(q,\mathcal{O}_{i+s})\times_{\mathcal{O}_{i+s}}\mathcal{M}(\mathcal{O}_{i+s},\mathcal{O}_{i})\right)
    \end{align} 
    and obtain a number, which then gives us the form 
    \begin{align}
        (u^{i+r}_{i,\partial})_*(\Psi')^*(\ell^{i+r}_i)^*\omega\in\Omega^{j-r+2}(\calo_{i+r},\mathcal{E}_{i+r})^G.
    \end{align}
\end{remark}

Finally, we are able to check $\partial^2 = 0$. Similar to \cite[Proposition 3.5]{Austin1995}, we only need to show: 
\begin{proposition}\label{boundary map verified as austin braam}
    For each $r\in\mathbb{Z}_{\geqslant 0}$, 
    \begin{align}
        \partial_0\partial_r+\partial_1\partial_{r-1}+\cdots+\partial_{r-1}\partial_1+\partial_r\partial_0 = 0.
    \end{align}
\end{proposition}
\begin{proof}
    We follow the steps as in \cite[Proposition 3.5]{Austin1995}. Let $\Psi'$ be the same map as that in Corollary \ref{corollary of the above stokes}. For any $\omega\in\Omega^j(\calo_i,\mathcal{E}_i)^G$, we find
    \begin{align}
        \partial_r\partial_0\omega 
        =\ & (-1)^{j+1}(u_i^{i+r})_*\left(d(\ell^{i+r}_i)^*\omega)\right)   \nonumber\\
        =\ & (-1)^{j+1}d(u_i^{i+r})_*(\ell^{i+r}_i)^*\omega + (-1)^{j+1+j-r+2} (u^{i+r}_{i,\partial})_*(\Psi')^*(\ell^{i+r}_i)^*\omega\ (\text{by Corollary \ref{corollary of the above stokes}})   \nonumber\\
        =\ & -\partial_0\partial_r\omega + (-1)^{-r+1}(u^{i+r}_{i,\partial})_*(\Psi')^*(\ell^{i+r}_i)^*\omega. 
    \end{align}
By Lemma \ref{theorem austin braam orientation induced on the boundary adjusted} and the following commutative diagram $$\begin{tikzcd}
   &  & \calo_i 
\\
   & \arrow[dl,"u^{i+r}_{i,\partial}"'] \mathcal{M}(\calo_{i+r},\calo_{i+s})\times_{\calo_{i+s}}{\mathcal{M}}(\calo_{i+s}, \calo_i) \arrow[r,"q^r_s"] \arrow[ur,""] \arrow[d, "p^r_s"] & \mathcal{M}(\calo_{i+s}, \calo_i) \arrow[u, "\ell_i^{i+s}"'] \arrow[d,"u_i^{i+s}"] \\
  \calo_{i+r} &\arrow[l,"u_{i+s}^{i+r}"] \mathcal{M}(\calo_{i+r},\calo_{i+s}) \arrow[r,"\ell^{i+r}_{i+s}"]  & \calo_{i+s}
\end{tikzcd},$$
we write the second summand into
\begin{align}
      & (-1)^{-r+1}(u^{i+r}_{i,\partial})_*(\Psi')^*(\ell^{i+r}_i)^*\omega   \nonumber\\
      =\ & (-1)^{-r+1}\sum_{0<s<r}(-1)^{r-s-1}(u^{i+r}_{i+s})_*(p^r_s)_*(\Psi')^*(\ell^{i+r}_i)^*\omega \text{\ (by Lemma \ref{theorem austin braam orientation induced on the boundary adjusted})}   \nonumber\\
      =\ & (-1)^{-r+1}\sum_{0<s<r}(-1)^{r-s-1}(u^{i+r}_{i+s})_*(p^r_s)_*(q^r_s)^*(\ell^{i+s}_i)^*\omega   \nonumber\\
    =\ & (-1)^{-r+1}\sum_{0<s<r}(-1)^{r-s-1}(u^{i+r}_{i+s})_*(\ell_{i+s}^{i+r})^*(u^{i+s}_{i})_*(\ell_{i}^{i+s})^*\omega \ (\text{by the diagram})    \nonumber\\
    =\ & -\sum_{0<s<r}\partial_{r-s}\partial_s\omega.
\end{align}
Thus, we have 
\begin{align}
    \partial_0\partial_r+\partial_1\partial_{r-1}+\cdots+\partial_{r-1}\partial_1+\partial_r\partial_0 = 0
\end{align}
and then $\partial^2 = 0$. 
\end{proof}
The first half of Theorem \ref{thom smale cohomology theorem} is proved.

\section{Topological realization}\label{section of Quasi-isomorphism in topology}
In this section, we prove the second half of Theorem \ref{thom smale cohomology theorem} following the steps similar to 
\cite[Theorem 3.8]{Austin1995} and using the nonorientable Thom isomorphism \cite[Theorem 7.10]{bott2013differential}.

\begin{definition}\normalfont
    The $G$-invariant de Rham complex of $M$ is formed by 
    \begin{align}
        \Omega^k(M)^G\coloneqq\{\omega\in\Omega^k(M): g^*\omega = \omega \text{\ for\ all\ }g\in G\}
    \end{align}
together with the de Rham $d$. We let $H^k(M)^G$ be the $k$-th cohomology group of this complex, and 
\begin{align}
    \Omega^*(M)^G \coloneqq \oplus_{k = 0}^m\Omega^k(M)^G.
\end{align}
\end{definition} 

 By \cite[Theorem 1.28]{felix2008algebraicmodels}, $H^k(M)^G$ is isomorphic to the original de Rham cohomology group of $M$. 
Thus, to prove Theorem \ref{thom smale cohomology theorem}, we need a quasi-isomorphism between the $G$-invariant Thom-Smale complex and the $G$-invariant de Rham complex. 

Using the fiber bundle structure
    $\pi_i: W^u(\calo_i)\to \calo_i$
and the associated integration along fibers $(\pi_i)_*$ defined in \cite[Theorem 7.10]{bott2013differential}, we get a map
\begin{align}
    \Phi: \Omega^k(M)^G &\to C^k(M,f)^G    \nonumber\\
    \omega&\mapsto\sum_{i = 0}^k(\pi_i)_*\left(\left.\omega\right\vert_{\overline{W^u(\mathcal{O}_i)}}\right) \ \ (k = 0,1,\cdots,m).
\end{align}
\begin{proposition}\label{chain map verifications}
    The map $\Phi$ is a cochain map.  
\end{proposition}
\begin{proof}
    Similar to \cite[Lemma 3.6]{Austin1995}, we check that for any $\omega\in\Omega^k(M)^G$ and any $0\leqslant i\leqslant k$, 
\begin{align}
    (\pi_i)_*\left(\left.d\omega\right\vert_{\overline{W^u(\mathcal{O}_i)}}\right) = \sum_{j = 0}^i\partial_{j}(\pi_{i-j})_*\left(\left.\omega\right\vert_{\overline{W^u(\mathcal{O}_{i-j})}}\right)
\end{align}    
    for all $\omega\in\Omega^k(M)^G$. By \cite[Lemmata 2.6 and 3.3]{Austin1995}, we have a $G$-equivariant embedding 
    \begin{align}
        \mathcal{P}: (-1,0]\times \bigcup_{0<j<i}\calm(\calo_i,\calo_{i-j})\times_{\calo_{i-j}} W^u(\calo_{i-j})\to\overline{W^u(\calo_i)}
    \end{align}
    onto a neighborhood of the codimension $1$ stratum of $\overline{W^u(\calo_i)}$. 
    Let $\mathcal{P}'$ be the restriction of $\mathcal{P}$ to 
    \begin{align}
        \{0\}\times \bigcup_{0<j<i}\calm(\calo_i,\calo_{i-j})\times_{\calo_{i-j}} W^u(\calo_{i-j}).
    \end{align}
    Similar to the proof of Corollary \ref{corollary of the above stokes}, we find that 
    \begin{align}
        & (\pi_i)_*\left(\left.d\omega\right\vert_{\overline{W^u(\mathcal{O}_i)}}\right)   \nonumber\\
        =\ & d \left((\pi_i)_*\left(\omega|_{\overline{W^u(\calo_i)}}\right)\right)+(-1)^{k-i+1}\left(\pi_{i,\partial}\right)_*(\mathcal{P}')^*\left(\omega|_{\overline{W^u(\calo_i)}}\right)   \nonumber\\
        =\ & \partial_0(\pi_i)_*\left(\omega|_{\overline{W^u(\calo_i)}}\right)+(-1)^{k-i+1}\left(\pi_{i,\partial}\right)_*(\mathcal{P}')^*\left(\omega|_{\overline{W^u(\calo_i)}}\right),
    \end{align}
    where $(\pi_{i,\partial})_*$ is induced by the fiber bundle structure
    \begin{align}
        \pi_{i,\partial}: \bigcup_{0<j<i}\calm(\calo_i,\calo_{i-j})\times_{\calo_{i-j}} W^u(\calo_{i-j})\to \calo_{i}
    \end{align}
    of the codimension-1 stratum of $\overline{W^u(\mathcal{O}_i)}$. 
    By \cite[Lemma 3.4]{Austin1995} and the commutative diagram 
    $$\begin{tikzcd}
   & \arrow[dl,"\pi_{i,\partial}"'] \mathcal{M}(\calo_i,\calo_{i-j})\times_{\calo_{i-j}}W^u(\calo_{i-j}) \arrow[r,""] \arrow[d, ""] & W^u(\calo_{i-j})  \arrow[d,"\pi_{i-j}"] \\
  \calo_i &\arrow[l,"u_{i-j}^{i}"] \mathcal{M}(\calo_i,\calo_{i-j}) \arrow[r,"\ell^{i}_{i-j}"]  & \calo_{i-j} 
\end{tikzcd},$$
    we get 
\begin{align}
    & (-1)^{k-i+1}\left(\pi_{i,\partial}\right)_*\left(\omega|_{\overline{W^u(\mathcal{O}_i)}}\right)   \nonumber\\
    %=\ & (-1)^{k-i-1}\sum_{\alpha = 0}^{i-1} \left(u_\alpha^i\right)_*\text{proj}_*\left((\text{pr}_\alpha)^*\left(\omega|_{W^u(S_\alpha^\oriented)}\right)\right)\\
    =\ & (-1)^{k-i+1}\sum_{j = 1}^{i} (-1)^{j-1}\left(u_{i-j}^i\right)_*(\ell_{i-j}^i)^*(\pi_{i-j})_*\left(\omega|_{\overline{W^u(\calo_{i-j})}}\right)   \nonumber\\
    =\ & \sum_{j = 0}^i\partial_{j}(\pi_{i-j})_*\left(\left.\omega\right\vert_{\overline{W^u(\mathcal{O}_{i-j})}}\right).
\end{align}
Thus, $\Phi$ is a cochain map. 
\end{proof} 

Next, we show that $\Phi$ is a quasi-isomorphism. Similar to \cite[Section 3.3]{Austin1995}, without loss of generality, we assume that $f(\calo_i) = i$ for all $0\leqslant i\leqslant m$. Then, for any $r\in\mathbb{Z}_{\geqslant 0}$, we let 
\begin{align}
    M_r = f^{-1}\left(r-\frac{1}{2},+\infty\right),\ \ M\backslash M_r = f^{-1}\left(-\infty, r-\frac{1}{2}\right]
\end{align}
so that for the $G$-invariant de Rham complex $(\Omega^*(M)^G,d)$, we get a filtration
\begin{align}
    \cdots\subseteq\Omega_{c}^k(M_{r+1})^G\subseteq \Omega_{c}^k(M_r)^G\subseteq\cdots\subseteq \Omega_{c}^k(M_0)^G = \Omega^k(M)^G
\end{align}
of $\Omega^k(M)^G$ for each $k$. Here, $\Omega^k_{c}(M_r)^G$ means the space of smooth $G$-invariant compactly supported $k$-forms on $M_r$. Meanwhile, we let 
\begin{align}
    C^k_{r}(M,f)^G = \bigoplus_{i\geqslant r}\Omega^{k-i}(\calo_i, \mathcal{E}_i)^G.
\end{align}
For the $G$-invariant Thom-Smale complex $(C^*(M,f)^G,\partial)$, we get a filtration 
\begin{align}
    \cdots\subseteq C_{r+1}^k(M,f)^G\subseteq C_{r}^k(M,f)^G\subseteq\cdots\subseteq C_{0}^k(M,f)^G= C^k(M,f)^G
\end{align}
of $C^k(M,f)^G$ for each $k$ as well. Let 
\begin{align}
    \graded\Omega_r^k = \Omega^k_{c}(M_r)^G/\Omega^k_{c}(M_{r+1})^G,
\end{align}
and 
\begin{align}
    \graded C_r^k = C^k_{r}(M,f)^G/C^k_{r+1}(M,f)^G = \Omega^{k-r}(\calo_r, \mathcal{E}_r)^G.
\end{align}
Then, for each $r$, we get two cochain complexes
\begin{align}
  (\graded\Omega_r^*,d): 0\xlongrightarrow[]{}\graded\Omega^0_r\xlongrightarrow[]{d}\graded\Omega^1_r\xlongrightarrow[]{d}\cdots\xlongrightarrow[]{d}\graded\Omega^m_r\xlongrightarrow[]{}0
\end{align}
and 
\begin{align}
   (\graded C_r^*,d): 0\xlongrightarrow[]{}\graded C^0_r\xlongrightarrow[]{\partial = d}\graded C^1_r\xlongrightarrow[]{\partial = d}\cdots\xlongrightarrow[]{\partial = d}\graded C^m_r\xlongrightarrow[]{}0.
\end{align}
Let $H^k(\graded\Omega^*_r)$ and $H^k(\graded C_r^*)$ be their $k$-th cohomology groups respectively.

By \cite[Lemma 3.7]{Austin1995}, if for any $r$ and $k$, the cochain map $\Phi$ induces an isomorphism between $H^k(\graded\Omega^*_r)$
and 
$H^k(\graded C_r^*)$, then $\Phi$ must be a quasi-isomorphism between the $G$-invariant Thom-Smale complex and the $G$-invariant de Rham complex. In addition, we notice that
\begin{align}\label{de Rham twisted with coefficients}
H^k(\graded C_r^*) = H^{k-r}(\calo_r,\mathcal{E}_r)^G,
\end{align}
where the latter is the $(k-r)$-th $G$-invariant de Rham cohomology group of $\mathcal{O}_r$ with local coefficient $\mathcal{E}_r$. Thus, to check the quasi-isomorphism, we need the following proposition. 
\begin{proposition}
\label{morse bott smale cochain complex cohomology proof without analysis}
    For any $r$ and $k$, 
    \begin{align}
        \Phi: H^k(\graded\Omega_r^*)\to H^{k-r}(\calo_r,\mathcal{E}_r)^G
    \end{align}
    is an isomorphism. 
\end{proposition}
\begin{proof}
We let $V_r = M_r\cap W^u(\calo_r)$ and $H^k_{c}(V_r)^G$ be its $k$-th $G$-invariant compactly supported de Rham cohomology group. 
By the nonorientable Thom isomorphism \cite[Theorem 7.10]{bott2013differential}, we find that 
$\Phi$ factors into 
\begin{align}
    H^k(\graded\Omega_r^*)\xrightarrow{\text{by\ restriction}} H^k_{c}(V_r)^G\xrightarrow{\text{Thom\ isomorphism}} H^{k-r}(\calo_r,\mathcal{E}_r)^G,
\end{align}
where the Thom isomorphism is given by the integration along each fiber. 
Now, to finish the proof, we must show that the restriction 
\begin{align}
    H^k(\graded\Omega_r^*)\to H^k_{c}(V_r)^G
\end{align}
is an isomorphism. 

We recall that for any embedded submanifold $S\subseteq M$, there is a relative de Rham cohomology \cite[Section I.6]{bott2013differential} for the pair $(M,S)$. If $S$ is $G$-invariant, we get the $G$-invariant version of the relative de Rham cohomology. For $0\leqslant k\leqslant m$, we let 
\begin{align}
    \Omega^k(M,S)^G = \Omega^k(M)^G\oplus\Omega^{k-1}(S)^G
\end{align}
and extend the de Rham differentiation to  
\begin{align}
    d: \Omega^k(M,S)^G&\to\Omega^{k+1}(M,S)^G   \nonumber\\
    (\alpha,\beta)&\mapsto (d\alpha, \alpha|_S-d\beta)
    \end{align}
for the relative case. As in \cite[Section 6.7]{bott2013differential}, this defines the $G$-invariant de Rham complex of $(M,S)$, and we let $H^k(M,S)^G$ be the $k$-th cohomology group of this complex. 

We consider the following two short exact sequences 
\begin{align}
    0\to\Omega_{c}^k(M_{r+1})^G\to\Omega_{c}^k(M_r)^G\to \graded\Omega^k_r\to 0
\end{align}
and 
\begin{align}
    0\to\Omega^k(M, M\backslash M_{r+1})^G\to\Omega^k(M, M\backslash M_{r})^G\to \Omega^k(M\backslash M_{r+1},M\backslash M_r)^G\to 0.
\end{align}
By \cite[Theorem 2.16]{allenhatcher2002}, they induce long exact sequences of cohomology groups $H^k_{c}(\cdot)^G$ and $H^k(\cdot)^G$ respectively. Then, we have the commutative diagram
$$
\begin{tikzcd}
  \cdots \to H^k_{c}(M_{r+1})^G \arrow[d, "\ "] \arrow[r] & H^k_{c}(M_{r})^G \arrow[d, "\ "] \arrow[r] & H^k(\graded\Omega^*_r) \arrow[d, "\ "] \to\cdots \\
  \cdots \to H^k(M, M\backslash M_{r+1})^G \arrow[r] & H^k(M, M\backslash M_{r})^G \arrow[r] & H^k(M\backslash M_{r+1},M\backslash M_r)^G\to\cdots
\end{tikzcd}
$$
between two long exact sequences. Since for any $k$ and $r$, 
\begin{align}
    H^k_{c}(M_{r})^G\to H^k(M, M\backslash M_{r})^G
\end{align}
is an isomorphism by the definition of the $G$-invariant relative de Rham complex, the map 
\begin{align}
    H^k(\graded\Omega^*_r)\to H^k(M\backslash M_{r+1},M\backslash M_r)^G
\end{align}
is an isomorphism according to the five lemma (See the proof of \cite[Theorem 2.27]{allenhatcher2002}). Now, we obtain 
\begin{align}
    H^k(M\backslash M_{r+1},M\backslash M_r)^G\cong H^k((M\backslash M_{r})\cup V_r, M\backslash M_r)^G
\end{align}
using the deformation retraction along the flow lines of $\nabla f$. 
Then, by excision, we find 
\begin{align}
    H^k((M\backslash M_{r})\cup V_r, M\backslash M_r)^G\cong H^k(V_r, \partial V_r)^G\cong H^k_{c}(V_r)^G.
\end{align}
Thus, the restriction map $H^k(\graded\Omega_r^*)\to H^k_{c}(V_r)^G$ is an isomorphism. 
\end{proof}
The second half of Theorem \ref{thom smale cohomology theorem} is thus a corollary of Proposition \ref{morse bott smale cochain complex cohomology proof without analysis}. 

\section{Deformed Hodge Laplacian and Hodge theory}\label{section of Witten Laplacians and Hodge theory}
In this section, we briefly review the $G$-invariant version of Hodge theory and construct the $G$-invariant Witten instanton complex associated with $f$. 

\begin{remark}\normalfont
    The results in this section only require the metric $\langle\cdot,\cdot\rangle$ and the function $f$ to be $G$-invariant. The assumptions \ref{morse bott orbit assumption 1} and \ref{morse bott orbit assumption 2} are required in all other sections. 
\end{remark}

It is straightforward to check that for any $\omega\in\Omega^*(M)$, 
\begin{align}
    d_{T}\omega = e^{-Tf}d\left(e^{Tf}\omega\right) = d\omega+Tdf\wedge\omega,
\end{align}
and thus 
\begin{align}
    d_T^*\omega = d^*\omega + T\nabla f\lrcornerwiththings\omega.
\end{align}
For any vector field $X$ and form $\omega$ on $M$, we define $X^*\coloneqq \langle X,\cdot\rangle$, 
\begin{align}
    \hat{c}(X)\omega \coloneqq X^*\wedge\omega+X\lrcornerwiththings\omega,
\end{align}
and 
\begin{align}
    c(X)\omega \coloneqq X^*\wedge\omega - X\lrcornerwiththings\omega.
\end{align}
Since $f$ and $\langle\cdot,\cdot\rangle$ are both $G$-invariant, the Dirac type operator
\begin{align}
    D_{T}\coloneqq d_{T}+d^*_{T} = d+d^*+T\hat{c}(df): \Omega^{\text{odd/even}}(M)^G\to\Omega^{\text{even/odd}}(M)^G
\end{align}
is well-defined on $G$-invariant forms on $M$. Recall that $D_{T}^2$ is called the deformed Hodge Laplacian.

Let $H^k_T(M)^G$ be the $k$-th cohomology group of the $G$-invariant de Rham complex
\begin{align}\label{witten twisted complex}
    0\xrightarrow[]{}\Omega^0(M)^G\xrightarrow[]{d_{T}}\cdots\xrightarrow[]{d_{T}}\Omega^{m}(M)^G\xrightarrow[]{}0.
\end{align}
For the same reason as \cite[Proposition 5.3]{wittendeformationweipingzhang}, an isomorphism 
\begin{align}
    H^k_T(M)^G\cong H^k(M)^G
\end{align}
is induced by the map $\omega\mapsto e^{Tf}\omega$. 
Moreover, (\ref{witten twisted complex}) admits a $G$-invariant Hodge decomposition.

\begin{proposition}\label{invariant Hodge decomposition}
    For each $0\leqslant k\leqslant m$, $\Omega^k(M)^G$ has an orthogonal decomposition 
    \begin{align}
        \ker(D_{T}^2: \Omega^k(M)^G\to\Omega^k(M)^G)\oplus\image(D_{T}^2: \Omega^k(M)^G\to\Omega^k(M)^G)
    \end{align}
    with respect to the $L^2$-norm on $\Omega^k(M)^G$ induced by $\langle\cdot,\cdot\rangle$. Moreover, we have an isomorphism 
    \begin{align}
        \ker(D_{T}^2: \Omega^k(M)^G\to\Omega^k(M)^G)&\cong H^k_T(M)^G   \nonumber\\
        \omega&\mapsto \text{the equivalence class of\ } \omega
    \end{align}
    between the kernel and the cohomology group.
\end{proposition}
\begin{proof}
    By applying the Gårding inequality \cite[Lemma 10.4.8]{nicolaescu2020lectures} and the Poincar\'e inequality \cite[Lemma 10.4.9]{nicolaescu2020lectures} to $G$-invariant forms, the proof is the same as the original Hodge decomposition.
\end{proof}
Recall that 
\begin{align}
    F^k_T(M,f,\alpha)^G = \bigoplus_{0\leqslant\lambda\leqslant \alpha}\left\{\omega\in\Omega^k(M)^G: D_{T}^2\omega = \lambda\omega\right\}
\end{align}
and that the $G$-invariant Witten instanton complex is given by
\begin{align}\label{witten instanton complex}
    0\xrightarrow[]{}F^0_T(M,f,\alpha)^G\xrightarrow[]{d_{T}}\cdots\xrightarrow[]{d_{T}}F^{m}_T(M,f,\alpha)^G\xrightarrow[]{}0.
\end{align}
    By Proposition \ref{invariant Hodge decomposition}, the $G$-invariant Witten instanton complex computes the de Rham cohomology of $M$: 
\begin{corollary}
    The two complexes {\normalfont (\ref{witten twisted complex})} and {\normalfont (\ref{witten instanton complex})} computes the same cohomology. 
\end{corollary}

 We have obtained several $G$-invariant complexes computing the de Rham cohomology of $M$. However, only (\ref{witten instanton complex}) gives the analytic realization that we want in Theorem \ref{isom given under a good metric be careful}.

\section{Metrics and connections}\label{sections of metrics and connections}
In this section, we complete the following two preparations for the estimates of $D_{T}$ and $D_T^2$: 
\begin{enumerate}[label = (\arabic*)]
    \item Adjust the original metric $\langle\cdot,\cdot\rangle$ according to the $G$-action and the $G$-equivariant Morse-Bott lemma \cite[Lemma 4.1]{wassermanequivariant} around each critical orbit. 
    \item Figure out how the Levi-Civita connection associated with the adjusted $\langle\cdot,\cdot\rangle$ acts on local frames around each critical orbit. 
\end{enumerate}

Given any critical orbit $\calo$ with dimension $ = n$ and Morse index $ = i$, we let $N$ (resp. $N(\varepsilon)$) be the normal bundle of $\calo$ defined with respect to $\langle\cdot,\cdot\rangle$ (resp. collection of all $v\in N$ satisfying $\langle v,v\rangle<\varepsilon^2$). According to \cite[Lemma 4.1]{wassermanequivariant}, we obtain: 
\begin{lemma}\label{morse-bott lemma equivariant without using local coordinates}
    There is a sufficiently small $\varepsilon>0$, a $G$-equivariant bundle map
    \begin{align}
        P: N\to N
    \end{align} 
    being an orthogonal projection with respect to $\langle\cdot,\cdot\rangle$ on each fiber, and a $G$-equivariant embedding 
    \begin{align}
        \varrho: N(8\varepsilon)\to M
    \end{align}
    satisfying that $\varrho|_\calo = \text{\normalfont id}$, such that for all vectors $v\in N(8\varepsilon)$, 
    \begin{align}\label{morse bott expression local but can be extended to global}
    f\circ\varrho(v) = f(\calo)-\dfrac{1}{2}\left\langle Pv,Pv\right\rangle+\dfrac{1}{2}\left\langle(1-P)v,(1-P)v\right\rangle.
    \end{align}
\end{lemma}
In addition, given a point $p\in\calo$, we let $G_p$ be its stabilizer, and $N_p$ be the fiber of $N$ at $p$. If we view $G$ as a principal $G_p$-bundle over $G/G_p$, we get an associated vector bundle $G\times_{G_p} N_p$
over $G/G_p$. Since $G/G_p\cong \calo$, we rephrase \cite[Theorem 1.25]{meinrenken}: 
\begin{lemma}\label{slice in our case}
    There is a $G$-equivariant bundle isomorphism $G\times_{G_p} N_p\cong N$. 
\end{lemma}
With Lemma \ref{morse-bott lemma equivariant without using local coordinates} and Lemma \ref{slice in our case}, we identify $N$ with $G\times_{G_p} N_p$. Now, we construct a Riemannian metric on $\llangle\cdot,\cdot\rrangle$ on $G\times_{G_p} N_p$ as in \cite[Corollary 1.27]{meinrenken} in the following steps:

\begin{enumerate}[label = (\arabic*)]
    \item We let $\mathbbm{1}$ be the identity of $G$. With respect to the adjoint representation of $G$ on $\mathfrak{g}$, we assign $\mathfrak{g}$ a $G$-invariant inner product $\langle\cdot,\cdot\rangle_{\mathfrak{g}}$. 
    \item Under $\langle\cdot,\cdot\rangle_{\mathfrak{g}}$, we let $U$ be the orthogonal complement of the Lie algebra of $G_p$ and choose an orthonormal basis $e_1,\cdots,e_n$ of $U$. 
    \item We let $v_1,\cdots,v_i,v_{i+1},\cdots,v_{m-n}$ be an orthonormal basis of $N_p$ with respect to the original bundle metric $\langle\cdot,\cdot\rangle$ such that $v_1,\cdots,v_{i}$ is the image of $P|_{N_p}$.
    \item At each equivalence class $[g,v]\in G\times_{G_p} N_p$, we let the following tangent vectors 
    \begin{align}
        w_k[g,v] =\ & \dfrac{d}{dt}\Big\vert_{t = 0}[g\exp(te_k),v] \ \ (1\leqslant k\leqslant n),\\
        z_\ell[g,v] =\ & \dfrac{d}{dt}\Big\vert_{t = 0}[g,v+tv_\ell]\ \ (1\leqslant \ell \leqslant m-n)
    \end{align}
    be an orthonormal basis of the tangent space $T_{[g,v]}(G\times_{G_p}N_p)$ with respect to $\llangle\cdot,\cdot\rrangle$. 
\end{enumerate}
Moreover, we extend this $\llangle\cdot,\cdot\rrangle$ to the whole manifold $M$ by using $G$-invariant bump functions so that the Riemannian metric equals $\llangle\cdot,\cdot\rrangle$ inside $N(4\varepsilon)$, while it is still $\langle\cdot,\cdot\rangle$ outside $N(6\varepsilon)$. 
\begin{proposition}
    This $\llangle\cdot,\cdot\rrangle$ is well-defined and $G$-invariant. 
\end{proposition}
\begin{proof}
    We know that $[gh,h^{-1}v] = [g,v]$ when $h\in G_p$. Then, we have 
    \begin{align}
        w_k[g,v] = \dfrac{d}{dt}\Big\vert_{t = 0}[g\exp(t\adjoint_h e_k),v]\ \ (1\leqslant k\leqslant n),\\ 
        z_\ell[g,v] = \dfrac{d}{dt}\Big\vert_{t = 0}[g,v+thv_\ell]\ \ (1\leqslant \ell\leqslant m-n).
    \end{align}
    Since both $\langle\cdot,\cdot\rangle$ and $\langle\cdot,\cdot\rangle_{\mathfrak{g}}$ are $G$-invariant, the choices of representatives for $[g,v]$ do not affect the metric $\llangle\cdot,\cdot\rrangle$.  
\end{proof}

\begin{remark}\normalfont
     Since we already assume that the original metric $\langle\cdot,\cdot\rangle$ satisfies the transversality conditions in \ref{morse bott orbit assumption 2}, by \cite[Appendix B]{Austin1995} and the continuous dependence on parameters of ODEs, the new metric $\llangle\cdot,\cdot\rrangle$ still satisfies \ref{morse bott orbit assumption 2} when $\varepsilon>0$ is sufficiently small. 
\end{remark}

Next, we describe the local frame of the adjusted metric.

\begin{notation}\label{new metric notation unchanged}\normalfont
    We will use the new metric $\llangle\cdot,\cdot\rrangle$ on $M$ in the rest of this paper. For simplicity, this new metric will still be denoted by  $\langle\cdot,\cdot\rangle$. In other words, the new metric replaces the old one. 
\end{notation} 

For any $g\in G$, we construct a local trivialization of $N$ around $[g,0]\in G\times_{G_p}\{0\} = \mathcal{O}$. By \cite[Theorem 9.3.7(iii)]{2011structureandgeometryofliegroups}, there is an open disk $\widetilde{U}\subset U$ centering at $0$ such that 
\begin{align}
   \phi: \widetilde{U}\times G_p&\to G   \nonumber\\
    (e,h)&\to \exp(e)h
\end{align}
is a diffeomorphism onto an open submanifold of $G$. 
Given any $g\in G$, we get a local trivialization
\begin{align}
    \varphi_g: \widetilde{U}\times N_p&\to N   \nonumber\\
    (x_1e_1+\cdots+x_ne_n,y_1v_1\cdots+y_{m-n}v_{m-n})&\mapsto [g\exp(x_1e_1+\cdots+x_ne_n),y_1v_1\cdots+y_{m-n}v_{m-n}]
\end{align}
and also a coordinate system $(\mathbf{x},
\mathbf{y}) = (x_1,\cdots,x_n,y_1,\cdots,y_{m-n})$ of $N$. By the compactness of $G$, finitely many $\varphi_g$'s cover the whole $N$. 

For $\mathbbm{1}\in G$ the identity element, we write $\varphi_{\mathbbm{1}}$ as $\varphi$ for convenience. 

Since $G$ can be nonabelian, $w_k\ (1\leqslant k\leqslant n)$ and $z_\ell\ (1\leqslant \ell\leqslant m-n)$ do not define a global frame of $TN$. However, we can still define a local frame. For example, on the local chart $\varphi$ of $N$ around $[\mathbbm{1},0]$, with the two projections
\begin{align}
    \pi_{\widetilde{U}}: \widetilde{U}\times G_p\to \widetilde{U}
\ \ \text{and}\ \ \pi_{G_p}: \widetilde{U}\times G_p\to G_p,
\end{align}
we write 
$u_k(t) = \exp(x_1e_1+\cdots+x_ne_n)\exp(te_k)$ and get a smooth local orthonormal frame $\{w_k, z_\ell\}_{k,l}: $
\begin{align}
     w_k(\mathbf{x},\mathbf{y}) = \dfrac{d}{dt}\Big\vert_{t = 0}\Big(\pi_{\widetilde{U}}\circ\phi^{-1}(u_k(t)),\left(\pi_{G_p}\circ\phi^{-1}(u_k(t))\right)\cdot(y_1v_1+\cdots y_{m-n}v_{m-n})\Big) \nonumber\\
    (1\leqslant k\leqslant n), \label{local frame local expressions}\\ 
z_\ell(\mathbf{x},\mathbf{y}) 
= \dfrac{d}{dt}\Big\vert_{t = 0}\left(x_1e_1+\cdots+x_ne_n,y_1v_1+\cdots+(y_\ell+t)v_l+\cdots+y_{m-n}v_{m-n}\right) \nonumber\\
(1\leqslant \ell\leqslant m-n). 
\end{align}
These $w_k$'s and $z_\ell$'s are orthonormal with respect to the new metric. Without loss of generality, we assume that they are oriented on $N$. Things are similar on any other $\varphi_g$ ($g\in G$).

Now, on the chart $\varphi$, we use the local coordinate system $(\mathbf{x},\mathbf{y})$ and write
\begin{align}
    w_k(\mathbf{x},\mathbf{y}) =\ & \sum_{j = 1}^n a_{kj}(\mathbf{x})\dfrac{\partial}{\partial x_j}+\sum_{s = 1}^{m-n}\left(\sum_{r = 1}^{m-n} b^r_{ks}(\mathbf{x}) y_r\right)\dfrac{\partial}{\partial y_s}, \\ 
z_\ell(\mathbf{x},\mathbf{y}) =\ & \dfrac{\partial}{\partial y_\ell}.
\end{align}
Here, $a_{kj}(\mathbf{x})$ and $b_{ks}^r(\mathbf{x})$ are functions in terms of $\mathbf{x}$. 

\begin{lemma}\label{skew symmetric b matrix}
    The functions $b_{ks}^r(\mathbf{x})$ satisfy that 
    \begin{align}
        b_{ks}^r(\mathbf{x}) = -b_{kr}^s(\mathbf{x}),
    \end{align}
     and that 
    $b_{ks}^r(\mathbf{x}) = 0$ when $1\leqslant s\leqslant i$ and $i+1\leqslant t\leqslant m-n$. 
\end{lemma}
\begin{proof}
    We know that the Lie algebra of the orthogonal group consists of skew-symmetric matrices. Since in (\ref{local frame local expressions}), $\pi_{G_p}\circ\phi^{-1}(u_k(t))$ acts isometrically, we then have
    \begin{align}
        b_{ks}^r(\mathbf{x}) = -b_{kr}^s(\mathbf{x}).
    \end{align}
    The second conclusion is because the space $N_p^-$ (resp. $N_p^+$) generated by $v_1,\cdots,v_i$ (resp. by $v_{i+1},\cdots,v_{m-n}$) is invariant under the $G_p$-action on $N_p$.  
\end{proof}

Finally, we study two connections induced by the adjusted metric. 

Let $\nabla^{TN}$ be the associated Levi-Civita connection on $TN$. Using the pullback induced by the inclusion $\mathcal{O}\xhookrightarrow{} N$, we get a connection $\nabla^{{TN}|_{\mathcal{O}}}$ on the vector bundle $TN|_\mathcal{O}$. 
Then, by the projection $\pi: N\to \mathcal{O}$ and the bundle isomorphism
\begin{align}
    \iota: TN & \to \pi^*(TN|_\mathcal{O})   \nonumber\\
    w_k[g,v]&\mapsto \left(w_k[g,0],[g,v]\right)   \nonumber\\
    z_\ell[g,v]&\mapsto \left(z_\ell[g,0],[g,v]\right),
\end{align}
we get a connection $\nabla$ on $TN$ defined by 
\begin{align}
    \iota\circ \nabla_XY = \left(\pi^*\nabla^{TN|_\mathcal{O}}\right)_X(\iota\circ Y)
\end{align}
for any $X, Y\in\mathfrak{X}(N)$.
Here, the notations of elements in the pullback bundle follow 
\begin{align}
    \pi^*(TN|_\mathcal{O})\subseteq TN|_\mathcal{O}\times N.
\end{align}
Let $a^{rs}$ be the function such that the matrix $\left[a^{rs}\right]_{1\leqslant r,s\leqslant n}$ is the inverse of $\left[a_{rs}\right]_{1\leqslant r,s\leqslant n}$. 
As in \cite[Lemma 2.2]{wenlumorsebottineq}, we can compute the difference between $\nabla^{TN}$ and $\nabla$. For simplicity, we define two notations 
\begin{align}
    \mathcal{A}_{jk\ell}(\mathbf{x}) \coloneqq & \dfrac{1}{2}\sum_{r,s = 1}^{n}\left(a_{j r}(\mathbf{x})\dfrac{\partial a_{\ell s}(\mathbf{x})}{\partial x_r}-a_{\ell r}(\mathbf{x})\dfrac{\partial a_{j s}(\mathbf{x})}{\partial x_r}\right)a^{sk}(\mathbf{x}), \\
    \mathcal{B}_{jk}^{\ell t}(\mathbf{x}) \coloneqq\ & \dfrac{1}{2}\sum_{r = 1}^n\sum_{s = 1}^{m-n}\left(a_{jr}(\mathbf{x})\dfrac{\partial b^t_{k\ell}(\mathbf{x})}{\partial x_r}-a_{kr}(\mathbf{x})\dfrac{\partial b^t_{j\ell}(\mathbf{x})}{\partial x_r} + b^t_{js}(\mathbf{x}) b^s_{k\ell}(\mathbf{x})-b^t_{ks}(\mathbf{x}) b^s_{j\ell}(\mathbf{x})\right). 
\end{align}
The precise computation is as follows. 
\begin{lemma}\label{local christoffel symbols}
On the local chart of $N$ given by $\varphi$, using the local frame $w_1,\cdots,w_n,z_1,\cdots,z_{m-n}$, we find:
\begin{align}
    \nabla^{TN}_{w_j}w_k =\ &\sum_{\ell = 1}^n \left(-\mathcal{A}_{jk\ell}(\mathbf{x}) - \mathcal{A}_{k\ell j}(\mathbf{x}) + \mathcal{A}_{\ell jk}(\mathbf{x})\right)w_\ell + \sum_{\ell, t = 1}^{m-n}\left(\mathcal{B}_{jk}^{\ell t}(\mathbf{x})+\sum_{q = 1}^n\mathcal{A}_{kqj}(\mathbf{x})b_{q\ell}^t(\mathbf{x})\right)y_t z_\ell    \nonumber\\
    =\ & \nabla_{w_j}w_k + \sum_{\ell, t = 1}^{m-n}\left(\mathcal{B}_{jk}^{\ell t}(\mathbf{x})+\sum_{q = 1}^n\mathcal{A}_{kqj}(\mathbf{x})b_{q\ell}^t(\mathbf{x})\right)y_t z_\ell\ ,\\
    \nabla^{TN}_{w_j}z_k =\ & \sum_{\ell = 1}^{m-n} b_{jk}^\ell(\mathbf{x})z_\ell + \sum_{\ell = 1}^n\sum_{t = 1}^{m-n}\left(\mathcal{B}_{\ell j}^{kt}(\mathbf{x})+\sum_{q = 1}^n \mathcal{A}_{jq\ell}(\mathbf{x})b_{qk}^t(\mathbf{x})\right)y_t w_\ell   \nonumber\\
    =\ & \nabla_{w_j}z_\ell + \sum_{\ell = 1}^n\sum_{t = 1}^{m-n}\left(\mathcal{B}_{\ell j}^{kt}(\mathbf{x})+\sum_{q = 1}^n \mathcal{A}_{jq\ell}(\mathbf{x})b_{qk}^t(\mathbf{x})\right)y_t w_\ell\ ,\\
    \nabla^{TN}_{z_j} w_k =\ & 0 + \sum_{\ell = 1}^n\sum_{t = 1}^{m-n} \mathcal{B}_{\ell k}^{jt}(\mathbf{x})y_t w_\ell = \nabla_{z_j}w_k + \sum_{\ell = 1}^n\sum_{t = 1}^{m-n} \mathcal{B}_{\ell k}^{jt}(\mathbf{x})y_t w_\ell\ ,\\
    \nabla^{TN}_{z_j}z_k =\ & \nabla_{z_j}z_k = 0.
\end{align}
\end{lemma}
\begin{proof}
    This is a direct computation using Koszul's formula \cite[Corollary 5.11(a)]{leeriemannian}. 
\end{proof}

\section{Spectral gap of the deformed Hodge Laplacian}\label{section of spectral gaps}
In this section, we figure out the spectral gap of $D_T$ when $T$ is sufficiently large. This spectral gap gives two disjoint intervals, separating horizontal and vertical eigenvalues. 

First of all, we clarify the behavior of $D_T$ on $G$-invariant forms around each critical orbit.
\begin{definition}\normalfont
    On $\Omega^*(N)$, we call the operators
    \begin{align}
       D^{H} = \sum_{k = 1}^n c(w_k)\nabla_{w_k} 
    \end{align}
    and 
    \begin{align}
        D^{V} = \sum_{\ell = 1}^{m-n} c(z_\ell)\nabla_{z_\ell}
    \end{align}
    the horizontal operator and the vertical operator respectively. 
    \end{definition}
    
    We see that $D^H$ and $D^V$ are independent of local orthonormal frames. Thus, they are globally defined on $N$. 
    As in \cite[(9.69)]{bismutandlebeau}, their supercommutator vanishes on $\Omega^*(N)$: 
\begin{proposition}\label{commutator}
    For any $\eta\in\Omega^*(N)$, we have 
    \begin{align}
        \left[D^H, D^V\right]\eta = 0.
    \end{align}
\end{proposition}
\begin{proof}
    The proof follows from Lemma \ref{skew symmetric b matrix}, Lemma \ref{local christoffel symbols}, and the fact that 
    \begin{align}
        c(Y)\nabla_X = \nabla_X c(Y) + c(\nabla_X Y)
    \end{align}
    for any smooth vector fields $X$ and $Y$ on $N$. 
\end{proof}
    
    Moreover, we notice that $\nabla$ is a metric connection. Let $\star$ be the Hodge star, and $dV = \star 1$ be the volume form on $N$. As in \cite[VIII(h)]{bismutandlebeau} and \cite[Section 2.4]{wenlumorsebottineq}, $D^H$ and $D^V$ are formal self-adjoint operators. 
\begin{proposition}\label{self adjoint D^H and D^V verifications}
    The operators $D^H$ and $D^V$ are formal self-adjoint on $\Omega^*(N)$. 
\end{proposition}
\begin{proof}
   For any $\omega, \eta\in\Omega^*(N)$ with one of them being compactly supported, 
    \begin{align}
         & \int_N \langle D^H\omega,\eta\rangle dV   \nonumber\\
        = & \int_N \sum_{k = 1}^n\langle\omega,\nabla_{w_k}(c(w_k)\eta)\rangle dV - \int_N\sum_{k = 1}^n w_k\langle\omega,c(w_k)\eta\rangle dV   \nonumber\\
        = & \int_N \sum_{k = 1}^n\langle\omega,c(w_k)\nabla_{w_k}\eta\rangle dV + \int_N \sum_{k = 1}^n \langle\omega,c\left(\nabla_{w_k} w_k\right)\eta\rangle dV - \int_N\sum_{k = 1}^n d\langle\omega,c(w_k)\eta\rangle\wedge\star w_k^*   \nonumber\\
        &\text{(Here, $w_k^*$ is the dual $1$-form $\langle w_k,\cdot\ \rangle$)}   \nonumber\\
        = & \int_N  \langle\omega,D^H\eta\rangle dV + \int_N \sum_{k = 1}^n \langle\omega,c\left(\nabla_{w_k} w_k\right)\eta\rangle dV - \int_N d\left(\sum_{k = 1}^n\langle\omega,c(w_k)\eta\rangle\wedge\star w_k^*\right)    \nonumber\\
        & - \int_N\sum_{k = 1}^n\langle\omega,c(w_k)\eta\rangle\wedge\star d^*w_k^*   \nonumber\\
        = & \int_N  \langle\omega,D^H\eta\rangle dV + \int_N \sum_{k = 1}^n \langle\omega,c\left(\nabla_{w_k} w_k\right)\eta\rangle dV + \int_N\sum_{k = 1}^n\sum_{\ell = 1}^n\left\langle\langle\omega,c(w_k)\eta\rangle, w_\ell\lrcorner \nabla^{TN}_{w_\ell}w_k^*\right\rangle dV    \nonumber\\
        & +\int_N\sum_{k = 1}^n\sum_{r = 1}^{m-n}\left\langle\langle\omega,c(w_k)\eta\rangle, z_r\lrcorner \nabla^{TN}_{z_r}w_k^*\right\rangle dV   \nonumber\\
        = & \int_N  \langle\omega,D^H\eta\rangle dV + \int_N \sum_{k = 1}^n \langle\omega,c\left(\nabla_{w_k} w_k\right)\eta\rangle dV - \int_N\sum_{\ell = 1}^n\sum_{k = 1}^n\left\langle\langle\omega,c(w_k)\eta\rangle, w_k^*\left(\nabla_{w_\ell}w_\ell\right)\right\rangle dV   \nonumber\\
        = & \int_N  \langle\omega,D^H\eta\rangle dV.
    \end{align}
    The second to last ``$=$'' is by Lemma \ref{local christoffel symbols}. The proof for $D^V$ is similar. 
\end{proof}

Since both $\nabla^{TN}$ and $\nabla$ are $G$-invariant connections (See \cite[five lines above (1.10)]{bgv} for the definition of $G$-invariant connections), we can also view $D_T$, $D^H$, and $D^V$ as operators on $\Omega^*(N)^G$. Recall that by \cite[(4.16)]{wittendeformationweipingzhang}, we have
\begin{align}\label{bochner type formula of D_T using o.n. frames}
    D_{T} = \sum_{k = 1}^n c(w_k)\nabla^{TN}_{w_k}+\sum_{\ell = 1}^{m-n}c(z_\ell)\nabla^{TN}_{z_\ell} + T\hat{c}(df)
\end{align}
on the local chart $\varphi$. Now, we hope to see how far away $D_T$ is from $D^H+D^V+T\hat{c}(df)$. 

Before we study the difference between $D_T$ and $D^H+D^V+T\hat{c}(df)$, we give two lemmas about the $G$-invariant forms on $N$. Recall that $n = \dim\mathcal{O}$, and $i$ is the Morse index of $\mathcal{O}$. 

\begin{notation}\normalfont
    We let $N_p^-\subseteq N_p$ (resp. $N^-\subseteq N$) be the subspace (resp. subbundle) generated by $v_1,\cdots,v_i$ (resp. equal to $G\times_{G_p} N_p^-$). Meanwhile, 
    we denote by $o(N^-)$ the orientation bundle of $N^-$. It identifies with the orientation bundle $\mathcal{E}_i\vert_{\mathcal{O}}$. 
\end{notation}

\begin{lemma}\label{local orientation valued form on critical orbits}
    For any $\omega\in\Omega^j(\mathcal{O}, o(N^-))^G$, on the local chart
    \begin{align}
        \varphi: \widetilde{V}\times\{0\}\to \mathcal{O},
    \end{align}
    we have
    \begin{align}
        \omega = \sum_{k_1<\cdots<k_j}c_{k_1\cdots k_j} w_{k_1}^*\wedge\cdots\wedge w_{k_j}^*\otimes e^-,
    \end{align}
    where each $c_{k_1\cdots k_j}$ is a constant, $w_1^*,\cdots, w_n^*$ are restricted on $\mathcal{O}$, and $e^-$ is the ``one-section'' of $o(N^-)$ with respect to the local trivialization 
    \begin{align}
        \varphi: \widetilde{V}\times N_p^-\to N^-.
    \end{align}
\end{lemma}
\begin{proof}
    Locally, we always have 
    \begin{align}
        \omega = \sum_{k_1<\cdots<k_j}c_{k_1\cdots k_j} {w}_{k_1}^*\wedge\cdots\wedge {w}_{k_j}^*\otimes e^-
    \end{align}
    and only need to show $c_{k_1\cdots k_j}(\mathbf{x}) = c_{k_1\cdots k_j}(0)$.
    
    In fact, for any $u = x_1e_1+\cdots+x_ne_n\in\widetilde{V}$, 
    \begin{align}
    &\ \sum_{k_1<\cdots<k_j}c_{k_1\cdots k_j}\left([\exp(u), 0]\right) {w}^*_{k_1[\exp(u), 0]}\wedge\cdots\wedge {w}^*_{k_j[\exp(u), 0]}\otimes (e^-)_{[\exp(u), 0]}    \nonumber\\
    = &\ \omega_{[\exp(u), 0]}   \nonumber\\
    = &\ (\exp(u)\cdot \omega)_{[\exp(u),0]}\ \ \text{(This is because $\omega$ is $G$-invariant)}   \nonumber\\
    = &\ \exp(u)\cdot \left(\omega_{[\mathbbm{1},0]}\right)   \nonumber\\
    = &\ \sum_{k_1<\cdots<k_j}c_{k_1\cdots k_j}\left([\mathbbm{1}, 0]\right) \exp(u)\cdot \left({w}^*_{k_1[\mathbbm{1}, 0]}\wedge\cdots\wedge {w}^*_{k_j[\mathbbm{1}, 0]}\otimes (e^-)_{[\mathbbm{1}, 0]} \right)   \nonumber\\
    = &\ \sum_{k_1<\cdots<k_j}c_{k_1\cdots k_j}\left([\mathbbm{1}, 0]\right) {w}^*_{k_1[\exp(u), 0]}\wedge\cdots\wedge {w}^*_{k_j[\exp(u), 0]}\otimes (e^-)_{[\exp(u), 0]}.
    \end{align}
    Thus, each $c_{k_1\cdots k_j}$ is a constant. 
\end{proof}
Lemma \ref{local orientation valued form on critical orbits} shows that 
\begin{align}
    \displaystyle\dim\Omega^j(\mathcal{O},o(N^-))^G\leqslant\binom{n}{j}
\end{align}
and is thus finite. 
\begin{corollary}
    The space $C^*(M,f)^G$ is finite dimensional.
\end{corollary}

By a similar procedure, we get a similar lemma on $N$. 
\begin{lemma}\label{invariant forms locally looks like}
    For any $\eta\in \Omega^*(N)^G$, on the local chart $\varphi: \widetilde{V}\times N_p\to N$, it has the form 
    \begin{align}
        \eta = \sum_{\substack{k_1<\cdots<k_r\\ \ell_1<\cdots<\ell_t}}C_{k_1\cdots k_r}^{\ell_1\cdots \ell_t}(\mathbf{y}) w_{k_1}^*\wedge\cdots\wedge w_{k_r}^*\wedge z_{\ell_1}^*\wedge\cdots\wedge z_{\ell_t}^*,
    \end{align}
    where each $C_{k_1\cdots k_r}^{\ell_1\cdots \ell_t}(\mathbf{y})$ is a function in terms of only the vertical coordinate $\mathbf{y}$. 
\end{lemma}

\begin{notation}\normalfont
    For any $\eta,\eta'\in\Omega^*(N)$ (resp. on $\Omega^*(M)$), we let $\abs{\eta} \coloneqq \langle\eta,\eta\rangle^{1/2}$, $\|\eta\|$ be the $L^2$-norm of $\eta$, and $(\eta,\eta')$ be the integral of $\eta\wedge\star\eta'$ on $N$ (resp. on $M$). In particular, we write $|\mathbf{y}|^2 = y_1^2+\cdots+y_{m-n}^2$. 
\end{notation} 

In addition, for the calculation on $N$, we need an extension of $f$ to resolve the issue that Lemma \ref{morse-bott lemma equivariant without using local coordinates} holds true only in a bounded neighborhood.  
\begin{notation}\normalfont
    Without loss of generality, when studying $D_T$ on $N(4\varepsilon)$, we view $f$ as a function given by (\ref{morse bott expression local but can be extended to global}) on the whole $N$ if necessary. In this way, we extend $D_T$ onto $\Omega^*(N)$. 
\end{notation}

As in \cite[Theorem 8.18]{bismutandlebeau} and \cite[Theorem 2.5]{wenlumorsebottineq}, on 
 $\Omega^*(N)^G$, we can now write $D_T$ into the sum ``horizontal'' $+$ ``vertical'' $+$ ``tail'' and describe how the ``tail'' acts on forms. 
\begin{proposition}
We let $R$ be the operator 
\begin{align}
    D_T - D^H - D^V - T\hat{c}(df).
\end{align}
Then, we have a constant $\Gamma_N>0$ such that  
   \begin{align}\label{estimate of tail R}
       \abs{R\eta} \leqslant \Gamma_N\cdot \abs{\mathbf{y}}\cdot \abs{\eta}
   \end{align}
    for all $\eta\in\Omega^*(N)^G$. Also, $R$ can be viewed as a matrix of order $O(\abs{\mathbf{y}})$ on the local chart $\varphi$. 
\end{proposition}
\begin{proof}
    By (\ref{bochner type formula of D_T using o.n. frames}), we find 
    \begin{align}\label{the precise expression of R}
        R = \sum_{k = 1}^n c(w_k)\left(\nabla^{TN}_{w_k} - \nabla_{w_k}\right) +  \sum_{\ell = 1}^{m-n} c(z_\ell)\left(\nabla^{TN}_{z_\ell} - \nabla_{z_\ell}\right).
    \end{align}
    The estimate (\ref{estimate of tail R}) follows from Lemma \ref{local christoffel symbols} and Lemma \ref{invariant forms locally looks like}. The matrix form of $R$ on the chart $\varphi$ is by letting $\mathbf{x} = 0$ in Lemma \ref{local christoffel symbols} according to Lemma \ref{invariant forms locally looks like}. 
\end{proof}

Now, we construct a space generated by ``approximate'' eigenforms of $D_T^2$. 

Let $\rho$ be a $G$-invariant bump function on $M$ such that $\rho \equiv 1$ on each $N(\varepsilon)$, and $\rho \equiv 0$ outside the union of all $N(2\varepsilon)$. 
With an unambiguous abuse of notations, we define  
\begin{align}\label{definition of J_T will be used later}
    J_T: \Omega^*(\mathcal{O},o(N^-))^G&\to \Omega^{*+i}(M)^G \nonumber\\
    \omega&\mapsto \rho \cdot \omega\wedge\exp\left(-\dfrac{T}{2}|\mathbf{y}|^2\right)z_1^*\wedge\cdots\wedge z_i^*.
\end{align}
The image of $J_T$ is globally defined because of the local coefficient $o(N^-)$.

\begin{notation}\normalfont\label{notation notation}
    As in \cite[(5.18)]{wittendeformationweipingzhang}, we let $E_{T}$ be the space spanned by all such $J_T\omega$ $\left(\omega\in C^*(M,f)^G \right)$, and $E_{T}^\perp$ be the $L^2$-orthogonal complement of $E_{T}$ in $\Omega^*(M)^G$. Meanwhile, we let $p_T$ and $p_T^\perp$ be the orthogonal projection from $\Omega^*(M)^G$ to $E_T$ and $E_T^\perp$ respectively. 
\end{notation}

We have $\nabla (z_1^*\wedge\cdots\wedge z_i^*) = 0$ by Lemma \ref{skew symmetric b matrix} and Lemma \ref{local christoffel symbols}. Then, we get a twisted version of $d+d^*$ on each critical orbit $\mathcal{O}$: 
\begin{align}\label{definition of the twisted d plus d star rigorous definition here}
    d+d^*: \Omega^j(\mathcal{O},o(N^-))^G &\to \Omega^{j\pm 1}(\mathcal{O},o(N^-))^G.
\end{align}
In fact, this $d+d^*$ twisted by the orientation bundle $o(N^-)$ is equal to $J_T^{-1}\circ D^H\circ J_T$.
In addition, using the local chart $\varphi$ and the notations in \ref{local orientation valued form on critical orbits}, we get an inner product on $C^*(M,f)$ given by the magnitude
\begin{align}\label{inner product local expression on each critical orbit making d+d^* self adjoint}
\abs{\omega} \coloneqq \sum_{k_1<\cdots<k_j}c_{k_1\cdots k_j}^2, \ \forall\ \omega\in\Omega^j(\mathcal{O},o(N^-)).
\end{align}

\begin{proposition}\label{twisted dirac type on critical orbits}
 On $C^*(M,f)^G$, $d+d^*$ is self-adjoint under the inner product {\normalfont(\ref{inner product local expression on each critical orbit making d+d^* self adjoint})}. 
\end{proposition}
\begin{proof}
    This is because $D^H$ is self-adjoint on $E_T$ by Proposition \ref{self adjoint D^H and D^V verifications}. 
\end{proof}

Recall that $\alpha_0$ is the spectral radius of $(d+d^*)^2$ on $C^*(M,f)^G$. 
\begin{proposition}\label{DT1 estimate nonzero}
    There exist $C_0, C_1, T_0 > 0$ and $\Gamma = \max_N\{\Gamma_N\}$ such that 
    \begin{align}
        \|D_{T}\eta\|\leqslant \left(\sqrt{\alpha_0} + C_0e^{-C_1T}+\Gamma\varepsilon\right)\|\eta\|
    \end{align}
    for all $T>T_0$ and all $\eta\in E_{T}$. 
\end{proposition}
\begin{proof}
    For any $\omega\in\Omega^j(\mathcal{O},o(N^-))^G$, applying 
    \begin{align}
       D_T = D^H+D^V+T\hat{c}(df)+R 
    \end{align}
    to it,  we find constants $C_0, C_1, T_0 > 0$ such that when $T>T_0$, 
    {\small\begin{align}
       \|D_{T}J_T\omega\|
       \leqslant \ & \left\|((d+d^*)\omega)\wedge\rho\exp\left(-\dfrac{T}{2}|\mathbf{y}|^2\right)z_1^*\wedge\cdots\wedge z_i^*\right\|    \nonumber\\
       & + \left\|c(d\rho) \omega\wedge\exp\left(-\dfrac{T}{2}|\mathbf{y}|^2\right)z_1^*\wedge\cdots\wedge z_i^*\right\| + \left\|RJ_T\omega\right\|   \nonumber\\
       \leqslant \ & \sqrt{\alpha_0} \|J_T\omega\|+C_0e^{-C_1T}\|J_T\omega\|+\Gamma_N\varepsilon\|J_T\omega\|.
    \end{align}}
    The last inequality is \text{by Proposition \ref{twisted dirac type on critical orbits} and (\ref{estimate of tail R})}.
\end{proof}

\begin{remark}\label{remark of small epsilon for adjustment}
\normalfont
    The $\Gamma$ in Proposition \ref{DT1 estimate nonzero} is independent of $\varepsilon$, meaning that whenever $\alpha>\alpha_0$, we can choose small $\varepsilon$ and assume that 
    \begin{align}
        \sqrt{\alpha_0}<\sqrt{\alpha_0} + C_0e^{-C_1T}+\Gamma\varepsilon<\sqrt{\alpha}.
    \end{align}
\end{remark}

With the local expression given by Lemma \ref{invariant forms locally looks like} for every $\eta\in\Omega^*(N)^G$, we define
\begin{align}
    \mathcal{H}_T \coloneqq\ & -\sum_{j = 1}^{m-n}\dfrac{\partial^2}{\partial y_j^2} - (m-n)T + T^2|\mathbf{y}|^2, \\
    \mathcal{L}_T \coloneqq\ & 2T\displaystyle\sum_{l = 1}^{i}z_l\lrcorner z_l^*\wedge+2T\sum_{l = i+1}^{m-n}z_l^*\wedge z_l\lrcorner.
\end{align}
According to \cite[Proposition 4.6]{wittendeformationweipingzhang} and Lemma \ref{invariant forms locally looks like}, we find on the local chart $\varphi$ that
\begin{align} \label{this is how the vertical laplacian works}
      (D^V+T\hat{c}(df))^2\eta 
   =\ & \sum_{\substack{k_1<\cdots<k_r\\ \ell_1<\cdots<\ell_t}} \mathcal{H}_T\left(C_{k_1\cdots k_r}^{\ell_1\cdots \ell_t}(\mathbf{y})\right)w_{k_1}^*\wedge\cdots\wedge w_{k_r}^*\wedge z_{\ell_1}^*\wedge\cdots\wedge z_{\ell_t}^* \nonumber\\
   & \ \ +  \sum_{\substack{k_1<\cdots<k_r\\ \ell_1<\cdots<\ell_t}} C_{k_1\cdots k_r}^{\ell_1\cdots \ell_t}(\mathbf{y})  \mathcal{L}_T\left(w_{k_1}^*\wedge\cdots\wedge w_{k_r}^*\wedge z_{\ell_1}^*\wedge\cdots\wedge z_{\ell_t}^*\right). 
\end{align}

\begin{lemma}\label{lemma of harmonic oscillators and linear parts on forms fiberwisely}
    When restricted to the space of $L^2$-sections of $\Lambda^*T^*N\vert_{N_p}$, the kernel of the positive operator $\mathcal{H}_T+\mathcal{L}_T$ is spanned by 
    \begin{align}
        \exp\left(-\dfrac{T}{2}\abs{\mathbf{y}}^2\right)w^*_{k_1}\wedge\cdots\wedge w^*_{k_j} \wedge z_1^*\wedge\cdots\wedge z_i^*\ \ (1\leqslant k_1<\cdots <k_j\leqslant n),
    \end{align}
    where each $w_k^*$ is restricted to $N_p$. 
    Meanwhile, all of its nonzero eigenvalues are $\geqslant 2T$. 
\end{lemma}
\begin{proof}
    The $\mathcal{H}_T$ is the harmonic oscillator on the space of $L^2$-functions. By \cite[Section 8.6 (6.12)]{taylor2023partialvol2}, we see that on $L^2$-functions, the kernel of $\mathcal{H}_T$ is generated by $\exp\left(-\dfrac{T}{2}\abs{\mathbf{y}}^2\right)$, and all nonzero eigenvalues are $\geqslant 2T$. 
    
    For $\mathcal{L}_T$, on the space of frames with constant coefficients, its kernel is generated by
    \begin{align}
        w^*_{k_1}\wedge\cdots\wedge w^*_{k_j} \wedge z_1^*\wedge\cdots\wedge z_i^*\ \ (1\leqslant k_1<\cdots <k_j\leqslant n),
    \end{align}
    while all nonzero eigenvalues are $\geqslant 2T$.
\end{proof}

Now, on the space $\Omega^*(N)^G$, we let $p_T'$ be the orthogonal projection from $\Omega^*(N)^G$ to the subspace $E_T'$ spanned by all 
\begin{align}
    \omega\wedge\exp\left(-\dfrac{T}{2}|\mathbf{y}|^2\right) z_1^*\wedge\cdots\wedge z_{i}^*\hspace{+1mm},\forall\ \omega\in \Omega^*(\mathcal{O},o(N^-))^G.
\end{align}
This $E_T'$ is exactly the kernel of $\mathcal{H}_T+\mathcal{L}_T$ on $\Omega^*(N)^G$. In addition, on the chart $\varphi$, with the notations from Lemma \ref{invariant forms locally looks like}, we find
{\begin{align}\label{definition of p_T' using density}
\begin{split}
 p_T'\eta =\ & \left(\dfrac{T}{\pi}\right)^{\frac{m-n}{2}}\exp\left(-\dfrac{T}{2}|\mathbf{y}|^2\right)\cdot\\
 &\sum_{\substack{k_1<\cdots<k_r}}\left(\int_{\mathbb{R}^{m-n}}C_{k_1\cdots k_r}^{1\cdots i}(\mathbf{y})\exp\left(-\dfrac{T}{2}|\mathbf{y}|^2\right)\left|d\mathbf{y}\right|\right) w_{k_1}^*\wedge\cdots\wedge w_{k_r}^*\wedge z_1^*\wedge\cdots\wedge z_i^*,
 \end{split}
\end{align}}
meaning that $p_T'$ is largely determined by the projection from $L^2(\mathbb{R}^{m-n})$ to the kernel of the harmonic oscillator (compare with \cite[(8.91)]{bismutandlebeau}). 

For the convenience of estimating $D_T$ on $E_T^\perp$, we present an auxiliary estimate of $p_T'$. 
\begin{lemma}
    There is a constant $C'>0$ such that when $T$ is sufficiently large, 
    \begin{align}\label{estimate of p_T' beforehand}
    \left\|p'_T\eta\right\| \leqslant C'T^{-1/4}\|\eta\|
    \end{align}
    for all $\eta\in\Omega^*(N)^G\cap E_T^\perp$ with $\supp(\eta)\subseteq N(4\varepsilon)$. 
\end{lemma}
\begin{proof}
    For any such $\eta$, as in \cite[(9.80)]{bismutandlebeau}, we can rewrite (\ref{definition of p_T' using density}) into 
    \begin{align}
     & p_T'\eta    \nonumber\\
 =\ & \left(\dfrac{T}{\pi}\right)^{\frac{m-n}{2}}\exp\left(-\dfrac{T}{2}|\mathbf{y}|^2\right)\cdot   \nonumber\\
 & \sum_{\substack{k_1<\cdots<k_r}}\left(\int_{\mathbb{R}^{m-n}}(1-\rho(\mathbf{y}))C_{k_1\cdots k_r}^{1\cdots i}(\mathbf{y})\exp\left(-\dfrac{T}{2}|\mathbf{y}|^2\right)\left|d\mathbf{y}\right|\right) w_{k_1}^*\wedge\cdots\wedge w_{k_r}^*\wedge z_1^*\wedge\cdots\wedge z_i^*.
\end{align}
Then, by 
\begin{align}
\left|(1-\rho(\mathbf{y}))C_{k_1\cdots k_r}^{1\cdots i}(\mathbf{y})\exp\left(-\dfrac{T}{2}|\mathbf{y}|^2\right)\right|\leqslant \left|C_{k_1\cdots k_r}^{1\cdots i}(\mathbf{y})\exp\left(-T\varepsilon^2/2\right)\right|, 
\end{align}
we obtain (\ref{estimate of p_T' beforehand}). 
\end{proof}
Immediately, we state the estimate of $D_T$ on $E_T^\perp$ as follows. 
\begin{proposition}\label{big DT4 estimate}
    There exist $C_2> 0$ and $T_1>0$ such that 
    \begin{align}
        \|D_{T}\eta\|\geqslant C_2\sqrt{T}\|\eta\|
    \end{align}
    for all $T>T_1$ and all $\eta\in E_T^\perp$. 
\end{proposition}
\begin{proof}
We prove it in three cases as \cite[Theorem 9.11]{bismutandlebeau} and \cite[Proposition 4.12]{wittendeformationweipingzhang}.

\vspace{+1mm}
\noindent\textbf{Case I}. We first assume that $\eta$ is supported in one $N(4\varepsilon)$. By (\ref{estimate of tail R}) and Proposition \ref{commutator}, 
    \begin{align}\label{estimate DT4 1}
        \|D_{T}\eta\|^2
        \geqslant\ & \dfrac{1}{2}\left\|D^V\eta+T\hat{c}(df)\eta + D^H\eta\right\|^2- \left\|R\eta\right\|^2 \nonumber\\
        \geqslant\ & \dfrac{1}{2}\left\|D^V\eta+T\hat{c}(df)\eta\right\|^2 + \dfrac{1}{2}\left\|D^H\eta\right\|^2- \Gamma_N^2\cdot (4\varepsilon)^2\cdot \left\|\eta\right\|^2.
    \end{align}
    By Lemma \ref{invariant forms locally looks like}, we write $\eta$ into an orthogonal decomposition $\eta_1+\eta_2$, with $\eta_1$ carrying $z_1^*\wedge\cdots\wedge z_i^*$, while $\eta_2$ carrying other combinations of $z_j^*$'s.  
    By (\ref{this is how the vertical laplacian works}), Lemma \ref{lemma of harmonic oscillators and linear parts on forms fiberwisely}, (\ref{definition of p_T' using density}), and (\ref{estimate of p_T' beforehand}), when $T$ is sufficiently large, we have
    \begin{align}\label{estimate DT4 2}
     \|D^V\eta+T\hat{c}(df)\eta\|^2 
        =\ & \left(\mathcal{H}_T\eta, \eta\right) + \left(\mathcal{L}_T\eta, \eta\right) \nonumber\\
       \geqslant\ & (\mathcal{H}_T\eta_1,\eta_1) + (\mathcal{L}_T\eta_2,\eta_2) \nonumber\\
         \geqslant\ & 2T\|\eta_1-p_T'\eta_1\|^2 + 2T\left\|\eta_2\right\|^2 \nonumber\\
         \geqslant\  & 2T\|\eta_1\|^2 - 2C'T^{1/2}\|\eta_1\|^2 + 2T\|\eta_2\|^2 \nonumber\\
         \geqslant\ & 2T\|\eta\|^2 - 2C'T^{1/2}\|\eta\|^2.
    \end{align}
    The same estimate holds when $\eta$ is supported in the union of all $N(4\varepsilon)$'s. 

    \vspace{+1mm}
    \noindent\textbf{Case II}. Next, we assume that $\eta$ is supported outside the union of all $N(2\varepsilon)$'s. Let $D \coloneqq d+d^*$ on $M$. Then, there is a constant $C''>0$ greater than the norm of the supercommutator $[D,\hat{c}(df)]$ on $\Omega^*(M)^G$. In addition, since there is another constant $C'''>0$ such that 
    \begin{align}
        |\nabla f|^2\geqslant C'''
    \end{align}
    outside the union of all $N(2\varepsilon)$'s, we obtain
    \begin{align}\label{estimate DT4 3}
        \left\|D_T\eta\right\|^2
        \geqslant T\left([D,\hat{c}(df)]\eta, \eta\right) + C'''T^2\|\eta\|^2 
        \geqslant (C'''T^2-C''T)\|\eta\|^2. 
    \end{align}
    \noindent\textbf{Case III}. Finally, for general $\eta\in E_T^\perp$, we unify the above estimates in exactly the same way as \cite[Proposition 4.12 Step 3]{wittendeformationweipingzhang}. 
    Let $\tilde{\rho}$ be the function on $M$ satisfying $\tilde{\rho}(\mathbf{y}) = \rho(\mathbf{y}/2)$ on each $N(4\varepsilon)$ and $\tilde{\rho} = 0$ outside the union of all $N(4\varepsilon)$'s. Then, we have
    \begin{align}
        \|D_T\eta\|
        \geqslant\ &\dfrac{1}{\sqrt{2}}\|(1-\tilde{\rho})D_T\eta\| + \dfrac{1}{\sqrt{2}}\|\tilde{\rho}D_T\eta\|   \nonumber\\
        =\ & \dfrac{1}{\sqrt{2}}\|D_T((1-\tilde{\rho})\eta)+[D,\tilde{\rho}]\eta\| + \dfrac{1}{\sqrt{2}}\|D_T(\tilde{\rho}\eta)+[\tilde{\rho},D]\eta\|   \nonumber\\
        \geqslant\ & \dfrac{1}{\sqrt{2}}\|D_T((1-\tilde{\rho})\eta)\|  + \dfrac{1}{\sqrt{2}}\|D_T(\tilde{\rho}\eta)\| - \dfrac{1}{\sqrt{2}}\|[D,\tilde{\rho}]\eta\| - \dfrac{1}{\sqrt{2}}\|[\tilde{\rho},D]\eta\|. 
    \end{align}
    Since $\tilde{\rho}\eta$ and $(1-\tilde{\rho})\eta$ are both in $E_T^\perp$, we apply (\ref{estimate DT4 1}) and (\ref{estimate DT4 2}) to $\|D_T(\tilde{\rho}\eta)\|$, (\ref{estimate DT4 3}) to $\|D_T((1-\tilde{\rho})\eta)\|$, and bounded operator norms to $\|[D,\tilde{\rho}]\eta\|$ and $\|[\tilde{\rho},D]\eta\|$.
\end{proof}

Finally, using Propositions \ref{DT1 estimate nonzero} and \ref{big DT4 estimate}, we present the spectral gap of $D_T^2$. The proof is adapted from the more general
    \cite[Lemma 5.3]{weipingzhangnewedition}, effective also for essential spectrum. 

\begin{proposition}\label{spectral gap of the Witten Laplacian proposition}
    When $T>0$ is sufficiently large, all eigenvalues of $D_{T}^2$ on $\Omega^*(M)^G$ are in 
    \begin{align}
        \left[0, (\sqrt{\alpha_0} + C_0 e^{-C_1T} + \Gamma\varepsilon)^2\right]\cup \left[C_2^2T,+\infty\right).
    \end{align}
\end{proposition}
\begin{proof}
    Suppose that $\omega\in\Omega^*(M)^G$ and 
    \begin{align}
        \left(\sqrt{\alpha_0} + C_0 e^{-C_1T} + \Gamma\varepsilon\right)^2<\lambda< C_2^2T
    \end{align}
    satisfy
    $D_{T}^2\omega = \lambda\omega$. 
    Then, we write $\omega = \omega_1+\omega_2$ with $\omega_1\in E_T$ and $\omega_2\in E_T^\perp$. Since $D_{T}$ is self-adjoint, by the above two results, we get 
    \begin{align}
        0 =\ & \left\langle (D_{T}^2-\lambda)\omega, \omega_1-\omega_2 \right\rangle    \nonumber\\
        =\ & \left\langle (D_{T}^2-\lambda)\omega_1,\omega_1 \right\rangle - \left\langle (D_{T}^2-\lambda)\omega_2,\omega_2 \right\rangle   \nonumber\\
        \leqslant\ & \left((\sqrt{\alpha_0} + C_0 e^{-C_1T} + \Gamma\varepsilon)^2-\lambda\right)\|\omega_1\|^2 + (\lambda - C_2^2T) \|\omega_2\|^2 \leqslant 0.
    \end{align}
    This shows that $\omega = 0$. 
\end{proof}

Let $P_{T}(\alpha)$ be the orthogonal projection from $\Omega^*(M)^G$ onto $F^*_{T}(M,f,\alpha)^G$. Then, we have $P_{T}(\alpha)\circ J_T$ mapping $E_T$ into the Witten instanton complex.
The following corollary shows that $J_T$ is ``almost'' the one-to-one correspondence.
\begin{corollary}\label{isomorphism without cochain operations}
    When $T$ is sufficiently large, $P_T(\alpha)\circ J_T$ is an isomorphism. 
\end{corollary}
\begin{proof}
     Using Proposition \ref{spectral gap of the Witten Laplacian proposition}, we find that for any $\omega\in C^*(M,f)^G$, 
    \begin{align}\label{estimate for injectivity}
    C_2\sqrt{T}\|J_T\omega - P_T(\alpha)J_T\omega\|
        \leqslant\ & \|D_T(J_T\omega - P_T(\alpha)J_T\omega)\|\nonumber\\
        \leqslant\ & \|D_TJ_T\omega\|+ \|D_TP_T(\alpha)J_T\omega\|\nonumber\\
        \leqslant\ & (\sqrt{\alpha_0}+C_0 e^{-C_1T} + \Gamma\varepsilon)\|J_T\omega\| + \sqrt{\alpha}\|P_T(\alpha)J_T\omega\|\nonumber\\
        \leqslant\ & (\sqrt{\alpha_0}+C_0 e^{-C_1T} + \Gamma\varepsilon+\sqrt{\alpha})\|J_T\omega\|.
    \end{align}
    Thus, $P_T(\alpha)J_T\omega = 0$ only when $\omega = 0$. Thus, $P_T(\alpha)\circ J_T$ is injective. 

    Next, if we have some $\eta\in F^*_T(M,f,\alpha)^G$ such that $\eta$ is $L^2$-orthogonal to the image of $P_T(\alpha)\circ J_T$, we show that $\eta = 0$ and therefore $P_T(\alpha)\circ J_T$ is surjective. Recall Notation \ref{notation notation}. In fact, by (\ref{estimate for injectivity}), 
    \begin{align}\label{estimate for pTperp}
        \|p_T^\perp\eta\|\nonumber
         =\ &\|\eta - p_T\eta\|\nonumber\\
        \geqslant\ & \|\eta - P_T(\alpha)p_T\eta\| - \|P_T(\alpha)p_T\eta - p_T\eta\|\nonumber\\
        \geqslant\ & \|\eta\| - C_2^{-1}T^{-1/2}(\sqrt{\alpha_0}+C_0 e^{-C_1T} + \Gamma\varepsilon+\sqrt{\alpha})\|p_T\eta\|\nonumber\\
        \geqslant\ & \|\eta\| - C_2^{-1}T^{-1/2}(\sqrt{\alpha_0}+C_0 e^{-C_1T} + \Gamma\varepsilon+\sqrt{\alpha})\|\eta\|.
    \end{align}
    The second to last line in (\ref{estimate for pTperp}) is also because $\eta$ is $L^2$-orthogonal to $P_T(\alpha)p_T\eta$. Thus, 
    \begin{align}\label{contradiction estimate}
       \sqrt{\alpha}\|\eta\|\nonumber
       \geqslant\ & \|D_T\eta\| \nonumber\\
        \geqslant\ & \|D_Tp_T^\perp\eta\| - \|D_Tp_T\eta\| \nonumber\\
        \geqslant\ & C_2\sqrt{T}\|p_T^\perp\eta\| - (\sqrt{\alpha_0} + C_0 e^{-C_1T} + \Gamma\varepsilon)\|p_T\eta\|\nonumber\\
        \geqslant\ & C_2\sqrt{T}\|\eta\| - 2(\sqrt{\alpha_0} + C_0 e^{-C_1T} + \Gamma\varepsilon)\|\eta\| - \sqrt{\alpha}\|\eta\|.
    \end{align}
    This (\ref{contradiction estimate}) is true only when $\eta = 0$. 
\end{proof}

\section{Analytic realization}\label{section of cochain isomorphisms}
In this section, we prove Theorem \ref{isom given under a good metric be careful}: When $T$ is sufficiently large, there is a cochain isomorphism between the $G$-invariant Thom-Smale complex and the $G$-invariant Witten instanton complex. We use the $P_{T}(\alpha)\circ J_T$ studied in Corollary \ref{isomorphism without cochain operations} as an auxiliary map. 

Before we carry out the analysis related to $P_T(\alpha)\circ J_T$, we follow \cite[Section X]{bismutandlebeau} to find a formal series. For convenience, we define a transformation
\begin{align}
    Q_T: \Omega^*(N)^G & \to \Omega^*(N)^G   \nonumber\\
    \omega_{[g,v]} & \mapsto \omega_{\left[g,v/\sqrt{T}\right]},
\end{align}
where $[g,v]\in N \cong G\times_{G_p} N_p$. 
Then, we find that on $\Omega^*(N)^G$, 
\begin{align}
    Q_TD_TQ_T^{-1} = D^H+\sqrt{T}\left(D^V+\hat{c}(df)\right)+\dfrac{1}{\sqrt{T}}R.
\end{align}
Here, $Q_T R Q_T^{-1} = \dfrac{1}{\sqrt{T}}R$ because by Lemma \ref{local christoffel symbols} and (\ref{the precise expression of R}), with respect to $\mathbf{y}$, the operator $R$ only contains terms of order $O(|\mathbf{y}|)$ (See \cite[(10.9)]{bismutandlebeau} for the situation when we have more curvature terms to handle.).

Correspondingly, we define 
\begin{align}
K: \Omega^*(\mathcal{O}, o(N^-))^G &\to \Omega^*(N)   \nonumber\\
    \omega&\mapsto \omega\wedge\exp\left(-\dfrac{1}{2}|\mathbf{y}|^2\right)z_1^*\wedge\cdots\wedge z_i^*.
\end{align}
Recall that $\alpha>\alpha_0$ in $F^*_T(M,f,\alpha)^G$. For each $K\omega$, the formal series is as follows. 

\begin{proposition}
    If $\delta\in\mathbb{R}$ and $\omega\in \Omega^*(\mathcal{O},o(N^-))^G$ satisfy $(d+d^*)\omega = \delta\cdot \omega$ for some $\delta\in\mathbb{R}$, then for any $\lambda\in\mathbb{C}$ satisfying $|\lambda| = \sqrt{\alpha}$, there is a formal power series 
    \begin{align}
        Z(\lambda,T) = \sum_{k = 0}^\infty \sigma_k(\lambda) T^{-k/2}
    \end{align}
    such that $Q_T(\lambda-D_T)Q_T^{-1}Z(\lambda,T) = K\omega$. 
    %{\color{red} $K$ is the $T = 1$/no $\rho$ version of $J_T$}
\end{proposition}
\begin{proof}We adapt the proof of \cite[(10.3)]{bismutandlebeau}. 
    Applying $Q_T(\lambda-D_T)Q_T^{-1}$ to $Z(\lambda,T)$, 
    we find:
    \begin{align}
    -(D^V+\hat{c}(df))\sigma_0(\lambda) &= 0,\label{first initial value}\\
        (\lambda-D^H)\sigma_0(\lambda) - (D^V+\hat{c}(df))\sigma_1(\lambda)& = K\omega,\label{second initial value}\\
        (\lambda-D^H)\sigma_{k+1}(\lambda) - (D^V+\hat{c}(df))\sigma_{k+2}(\lambda) - R\sigma_k(\lambda)& = 0, \ \forall\hspace{+1mm}k\in\mathbb{Z}_{\geqslant 0}.\label{iteration relation}
    \end{align}
    Let $q$ (resp. $q^\perp$) be the orthogonal projection from $\Omega^*(N)^G$ onto $\image K$ (resp. the orthogonal complement of $\image K$). 
    Then, by (\ref{first initial value}) and (\ref{second initial value}), we choose 
    \begin{align}
        \sigma_0(\lambda) = \dfrac{K\omega}{\lambda-\delta}\text{\ \ and\ \ } q^\perp\sigma_1(\lambda) = 0
    \end{align}
    as the initial values. 
    
    To find all $\sigma_k(\lambda)$ by (\ref{iteration relation}), first, we see that $\lambda-D^H$ is invertible on $\image K$ since $\abs{\lambda}>\sqrt{\alpha_0}$. Next, by Lemma \ref{lemma of harmonic oscillators and linear parts on forms fiberwisely}, 
    $(D^V+\hat{c}(df))^{-2}$
    is well-defined on the space of $L^2$-sections of $\Lambda^*T^*N\vert_{N_p}$: 
    \begin{enumerate}[label = (\arabic*)]
        \item On the kernel of $(D^V+\hat{c}(df))^2$ restricted to $\Lambda^*T^*N\vert_{N_p}$, we let 
    $(D^V+\hat{c}(df))^{-2}$ be $0$;
    \item On the orthogonal complement of this kernel, $(D^V+\hat{c}(df))^{-2}$ is given by the functional calculus \cite[Proposition 5.30]{roe1999elliptic}.
    \end{enumerate}
    In particular, since $D^V+\hat{c}(df)$ is $G$-equivariant, $(D^V+\hat{c}(df))^{-2}$ is well-defined on $\Omega^*(N)^G$. 
    
   Following \cite[(10.15)]{bismutandlebeau}, we write (\ref{iteration relation}) into two equations:  
    \begin{align}
    (\lambda - D^H)\left(q\sigma_{k+1}(\lambda)\right) - q\circ R(\sigma_k(\lambda)) =\ & 0, \label{modified series equation simple part}\\
    -(D^V+\hat{c}(df))(q^\perp\sigma_{k+2}(\lambda)) + (\lambda-D^H)(q^\perp\sigma_{k+1}(\lambda)) - q^\perp\circ R(\sigma_k(\lambda)) =\ & 0. \label{modified series equation}
    \end{align}
   Similar to \cite[(10.17)]{bismutandlebeau}, we apply $(\lambda-D^H)^{-1}$ to (\ref{modified series equation simple part}) and 
   \begin{align}
       (D^V+\hat{c}(df))^{-2}\circ (D^V+\hat{c}(df))
   \end{align}
   to (\ref{modified series equation}). Using $\sigma_0(\lambda)$ and $q^\perp\sigma_1(\lambda)$, we obtain a unique formal series $Z(\lambda,T)$ for $\omega$. 
\end{proof}

We now use the formal series $Z(\lambda,T)$ to get the following Proposition \ref{key estimate global L infinity} by adapting the proof of \cite[Theorem 10.1]{bismutandlebeau}. Proposition \ref{key estimate global L infinity}
     shows that $P_T(\alpha)\circ J_T$ is an isomorphism (but not a cochain isomorphism) for large $T$, describing the size of $P_T(\alpha)\circ J_T - J_T$ inside and outside the radius-$2\varepsilon$ tubular neighborhoods. 
Also, it is the key to prove Theorem \ref{isom given under a good metric be careful}. 
\begin{proposition}\label{key estimate global L infinity}
    For any $\mu\in\mathbb{N}$, there exist $C_3>0$ and $\Gamma'>0$ $(\Gamma'$ is irrelevant to $\varepsilon)$ such that when $T$ is sufficiently large, 
    \begin{align}
        \abs{P_T(\alpha)J_T\omega - J_T\omega}\leqslant C_3T^{-\mu/2}\left|\omega\right| + \Gamma'\cdot \varepsilon\cdot
    \left|J_T\omega\right|
    \end{align}
    for all $\omega\in \Omega^*(\mathcal{O},o(N^-))^G$. 
\end{proposition}
\begin{proof}
    Without loss of generality, we assume $(d+d^*)\omega = \delta\cdot\omega$ and $\abs{\omega} = 1$, i.e.,  
    \begin{align}
        \sum_{k_1<\cdots<k_j}c_{k_1\cdots k_j}^2 = 1
    \end{align}
    in the notation from Lemma \ref{local orientation valued form on critical orbits}.
    We recall that $|\lambda|^2 = \alpha>\alpha_0$ and let 
    \begin{align}
        Z_\ell(\lambda, T) =  \displaystyle\sum_{k=0}^{\ell+1}\sigma_k(\lambda)T^{-k/2}.
    \end{align}
    When $\varepsilon$ is sufficiently small, and $T$ is sufficiently large, by the spectral gap Proposition \ref{spectral gap of the Witten Laplacian proposition}, we apply the functional calculus \cite[(9.153)]{bismutandlebeau} to $P_T(\alpha)$ and get
    \begin{align}
        & P_T(\alpha)J_T\omega - J_T\omega \nonumber\\
         =\ & \dfrac{1}{2\pi\sqrt{-1}}\int_{\abs{\lambda}^2 = \alpha}\dfrac{1}{\lambda-D_{T}}J_T\omega d\lambda - J_T\omega  \nonumber\\
         =\ &\dfrac{1}{2\pi\sqrt{-1}}\int_{\abs{\lambda}^2 = \alpha}\dfrac{1}{\lambda-D_T}J_T\omega d\lambda - \dfrac{1}{2\pi\sqrt{-1}}\int_{\abs{\lambda}^2 = \alpha}\rho\cdot Q^{-1}_T Z_{\ell}(\lambda,T)d\lambda \nonumber\\
         & +\dfrac{1}{2\pi\sqrt{-1}}\int_{\abs{\lambda}^2 = \alpha}\rho\cdot Q^{-1}_T Z_{\ell}(\lambda,T)d\lambda - J_T\omega. \label{the projection using complex integral functional calculus}
    \end{align}
    Recall the bump function $\rho$ on $N$. We now estimate (\ref{the projection using complex integral functional calculus}) in two parts: 

    \vspace{+1mm}
    \noindent\textbf{Part I}: On the one hand, we find that 
    \begin{align}
        & (\lambda-D_T)\left(\rho Q_T^{-1}Z_\ell(\lambda,T)\right) - J_T\omega   \nonumber\\
        =\ & \rho(\lambda-D_T)Q_T^{-1}Z_\ell(\lambda,T) - c(d\rho)Q_T^{-1}Z_\ell(\lambda,T) - \rho Q_T^{-1}K\omega   \nonumber\\
        =\ & Q_T^{-1}{\Big[}Q_T(\rho)\cdot\left((\lambda-D^H)\sigma_{\ell+1}(\lambda)T^{-(\ell+1)/2}-R\sigma_\ell(\lambda)T^{-(\ell+1)/2}-R\sigma_{\ell+1}(\lambda)T^{-(\ell+2)/2}\right)   \nonumber\\
        & \ \ \ \ \ \ \ \ -c(Q_T(d\rho))Z_\ell(\lambda,T){\Big]}   \nonumber\\
        =\ & \rho\cdot\left[Q_T^{-1}((\lambda-D^H)\sigma_{\ell+1}(\lambda))
        -Q_T^{-1}(R\sigma_\ell(\lambda))-Q_T^{-1}(R\sigma_{\ell+1}(\lambda))T^{-1/2}\right]\cdot T^{-(\ell+1)/2}   \nonumber\\
        & - c(d\rho)(Q_T^{-1}Z_\ell(\lambda,T)).
    \end{align}
    Following \cite[(10.16)]{bismutandlebeau}, we consider the Schwartz semi-norms \cite[Part 4, Section 8]{handbookofglobalanalysis} of $\sigma_\ell(\lambda)$ and $\sigma_{\ell+1}(\lambda)$. For any $\nu\in\mathbb{N}$ and $\nu_1+\cdots+\nu_{m-n} = \nu$, there are $\zeta_\nu, \zeta'_\nu>0$ such that 
    \begin{align}
        \left|\partial^{\nu_1}_{y_1}\cdots\partial^{\nu_{m-n}}_{y_{m-n}}\sigma_\ell(\lambda)\right|\leqslant\dfrac{\zeta_\nu}{\left(1+|\mathbf{y}|\right)^\nu}
    \end{align}
    and 
    \begin{align}
        \left|\partial^{\nu_1}_{y_1}\cdots\partial^{\nu_{m-n}}_{y_{m-n}}\sigma_{\ell+1}(\lambda)\right|\leqslant\dfrac{\zeta'_\nu}{\left(1+|\mathbf{y}|\right)^\nu}
    \end{align}
    uniformly for $\abs{\lambda} = \sqrt{\alpha}$. Let $\|\cdot\|_\nu$ denote the $\nu$-th Sobolev norm on $M$. 
    Then, we have a constant $\widetilde{C}_\nu>0$ such that  when $T$ is sufficiently large,
    \begin{align}
          \|(\lambda-D_T)(\rho Q_T^{-1}s_\ell(\lambda,T))-J_T\omega\|_\nu  
        \leqslant \widetilde{C}_\nu\cdot T^{-(\ell+1)/2}\cdot T^{\nu/2}\cdot T^{-\text{rank}(N)/2}.
    \end{align}

    On the other hand, 
    by Proposition \ref{spectral gap of the Witten Laplacian proposition}, 
    we find that for all $\eta\in\Omega^*(M)^G$,
    \begin{align}\label{estimate 1 for sobolev norm nu}
    \|(\lambda-D_T)\eta\|\geqslant (\sqrt{\alpha}-\sqrt{\alpha_0} - C_0 e^{-C_1T} - \Gamma\varepsilon)\|\eta\|.
    \end{align}
    In addition, by \cite[(6.18)]{wittendeformationweipingzhang} or verifying inductively using the Gårding inequality of $D = d+d^*$ on $M$, we have a constant $C_4>0$ such that when $T$ is sufficiently large, 
    \begin{align}\label{estimate 2 for sobolev norm nu}
        \|\eta\|_{\nu+1}\leqslant C_4 T^{\nu+1}\left(\|(\lambda-D_T)\eta\|_\nu + \|\eta\|\right)
    \end{align}
    for all $\eta\in\Omega^*(M)^G$. 
    By (\ref{estimate 1 for sobolev norm nu}) and (\ref{estimate 2 for sobolev norm nu}), there is a constant $C_5>0$ such that 
    for any sufficiently large $T$ and any unit eigenform $\omega$ of the twisted $d+d^*$ on $C^*(M,f)^G$, 
    \begin{align}
        & \|\rho Q_T^{-1}Z_\ell(\lambda,T) - (\lambda-D_T)^{-1}J_T\omega\|_{\nu+1}  \nonumber\\
        \leqslant\ & C_4 T^{\nu+1}\left(\|(\lambda-D_T)\rho Q_T^{-1} Z_\ell(\lambda, T) - J_T\omega\|_\nu + \|\rho Q_T^{-1} Z_\ell(\lambda, T) - (\lambda-D_T)^{-1}J_T\omega\|\right)  \nonumber\\
        \leqslant\ & C_5 T^{\nu+1} \|(\lambda-D_T)\rho Q_T^{-1} Z_\ell(\lambda,T) - J_T\omega\|_\nu  \nonumber\\
        \leqslant\ & C_5 T^{\nu+1}\cdot \widetilde{C}_\nu\cdot T^{-(\ell+1)/2}\cdot T^{\nu/2}\cdot T^{-\text{rank}(N)/2}.\label{part 1 detailed L infinity estimate}
    \end{align}
    \vspace{+1mm}
    \noindent\textbf{Part II}: We look at the subtraction
    \begin{align}\label{subtraction do not forget}
        & \dfrac{1}{2\pi\sqrt{-1}}\int_{|\lambda|^2 = \alpha}\rho Q_T^{-1}Z_\ell(\lambda,T)d\lambda - J_T\omega \nonumber\\
         =\ & \dfrac{1}{2\pi\sqrt{-1}}\int_{|\lambda|^2 = \alpha} \rho Q_T^{-1}\left(\sigma_2(\lambda)T^{-1}+\sigma_3(\lambda)T^{-3/2}+\cdots+\sigma_{\ell+1}(\lambda)T^{-(\ell+1)/2}\right)d\lambda.
    \end{align}
    Recall that $\supp(\rho)\subseteq N(2\varepsilon)$. By checking the degree of Hermite polynomials \cite[Section 8.6]{taylor2023partialvol2} arising from the eigenfunctions of the harmonic oscillator, we find $\Gamma_k>0$ such that 
    \begin{align}
        \left|\dfrac{1}{2\pi\sqrt{-1}}\int_{|\lambda|^2 = \alpha}\rho Q_T^{-1}(\sigma_k(\lambda))T^{-k/2}d\lambda\right|\leqslant \Gamma_k\cdot(2\varepsilon)^k\cdot\rho\exp\left(-\dfrac{T}{2}|\mathbf{y}|^2\right)\label{part 2 detailed L infinity estimate}
    \end{align}
    for each $k = 1,\cdots,\ell+1$.  Here, $\Gamma_k$ is irrelevant to $\varepsilon$.
    
    Combining (\ref{part 1 detailed L infinity estimate}), (\ref{subtraction do not forget}), and (\ref{part 2 detailed L infinity estimate}), by \cite[Corollary 6.22]{warner2013foundations}, there is a $C_6>0$ such that
    \begin{align}
        \left|P_T(\alpha)J_T\omega - J_T\omega\right| \leqslant C_6 T^{-\left(\ell-3\nu-1+\text{rank}(N)\right)/2} + \sum_{k = 1}^{\ell+1}\Gamma_k\cdot(2\varepsilon)^k\cdot\rho\exp\left(-\dfrac{T}{2}|\mathbf{y}|^2\right). 
    \end{align}
    The proposition is verified after reorganizing all the constants. 
\end{proof}

\begin{remark}\label{remark also about small epsilons}
\normalfont
    As in Remark \ref{remark of small epsilon for adjustment}, the constant $\Gamma'$ is independent of the choice of $\varepsilon$. Thus, we can choose sufficiently small $\varepsilon$ so that $\Gamma'\cdot\varepsilon < 1/2$. 
\end{remark}

Finally, with Proposition \ref{key estimate global L infinity}, we can prove Theorem \ref{isom given under a good metric be careful} and establish the analytic realization.  
As in \cite[Definition 6.10]{Bismut1994} and \cite[Definition 6.8]{wittendeformationweipingzhang}, we define two automorphisms on the space $C^*(M,f)^G$. Let $\mathcal{F}_T$ be the linear map
\begin{align}\label{diagonal matrix by function values on critical set}
      \mathcal{F}_T: C^*(M,f)^G &\to C^*(M,f)^G  \nonumber\\
    \omega\in\Omega^j(\mathcal{O},o(N^-))^G&\mapsto e^{Tf(\mathcal{O})}\cdot\omega,
\end{align}
where $\mathcal{O}$ is a critical orbit. We immediately see $\mathcal{F}_T$ is an automorphism.%Then, we get an automorphism $e^{T\mathcal{F}}$ of $C^*(M,f)$. 

Meanwhile, we recall the coordinate $\mathbf{y} = (y_1,\cdots,y_{m-n})$ around each critical orbit $\mathcal{O}$, where the Morse index of $\mathcal{O}$ is $i$, and $\dim\mathcal{O} = n$. Using the same bump function $\rho$ as that in (\ref{definition of J_T will be used later}), the map
\begin{align}\label{main body integral but with epsilon adjust by error}
\begin{split}
\mathcal{N}_T: C^*(M,f)^G &\to C^*(M,f)^G\\
    \omega\in\Omega^j(\mathcal{O},o(N^-))^G&\mapsto \omega\cdot\int_{\mathbb{R}^i}\rho(y_1,\cdots,y_i,0,\cdots,0)e^{-T\left(y_1^2+\cdots+y_i^2\right)}dy_1\cdots dy_i
    \end{split}
\end{align}
is also an automorphism. 

\begin{proof}[Proof of Theorem \ref{isom given under a good metric be careful}]
Recall the map 
     \begin{align}
    \Phi_{T}: F^k_T(M,f,\alpha)^G &\to C^k(M,f)^G   \nonumber\\
    \eta&\mapsto\sum_{i = 0}^k(\pi_i)_*\left(\left.e^{Tf}\cdot\eta\right\vert_{\overline{W^u(\mathcal{O}_i)}}\right) \ \ (k = 0,1,\cdots,m)
\end{align}
between two cochain complexes. By Proposition \ref{chain map verifications}, this $\Phi_T$ is a cochain map.
Also, for any $0\leqslant r\leqslant k$ and any $\omega\in \Omega^{k-r}(\mathcal{O}_r,\mathcal{E}_r)^G$ satisfying $\abs{\omega} = 1$, by Proposition \ref{key estimate global L infinity}, we find
\begin{align}\label{main plus tail}
     \Phi_TP_T(\alpha)J_T\omega 
     = \Phi_T(J_T\omega + \tau) + \sum_{i = 0}^k\sum_{\mathcal{O}\subseteq\mathcal{O}_i}e^{Tf(\mathcal{O})}\cdot(\pi_i)_*\left(e^{T(f-f(\mathcal{O}))}\cdot\tau'\rvert_{\overline{W^u(\mathcal{O})}}\right),
\end{align}
where $\mathcal{O}$ is an orbit in the union $\mathcal{O}_i$ of all critical orbits having the same Morse index $i$, and $\tau, \tau'\in\Omega^*(M)^G$ satisfies that 
\begin{align}
    \abs{\tau}\leqslant \Gamma'\cdot \varepsilon\cdot\abs{J_T\omega}, \text{\ and\ } \abs{\tau'}\leqslant C_3T^{-\mu}.
\end{align}

We first look at the tail term in (\ref{main plus tail}) given by $\tau'$. Since for every critical orbit $\mathcal{O}\subseteq \mathcal{O}_i\ (0\leqslant i\leqslant k)$, there is $f-f(\mathcal{O})\leqslant 0$ on $\overline{W^u(\mathcal{O})}$, we then find a constant $\xi_{\mathcal{O}}>0$ such that
\begin{align}\label{very tail estimate}
    \left|(\pi_i)_*\left(e^{T(f-f(\mathcal{O}))}\cdot\tau'\rvert_{\overline{W^u(\mathcal{O})}}\right)\right|\leqslant \xi_{\mathcal{O}}\cdot C_3T^{-\mu}\ \ (\mathcal{O}\subseteq\mathcal{O}_i, 0\leqslant i\leqslant k)
\end{align}
when $T$ is sufficiently large. 

Next, we look at the main term 
\begin{align}\label{separate into three parts to finish main result proof}
     \Phi_T(J_T\omega+\tau) = \sum_{i = 0}^k \sum_{\mathcal{O}\subseteq\mathcal{O}_i}e^{Tf(\mathcal{O})}\cdot(\pi_i)_*\left(\left.e^{T(f-f(\mathcal{O}))}\cdot (J_T\omega+\tau)\right\vert_{\overline{W^u(\mathcal{O})}}\right)
\end{align}
in (\ref{main plus tail}). It separates into three parts: 

\vspace{+1mm}
\noindent \textbf{Part I}: When $i = r$, we write $\omega$ uniquely into the sum
$\displaystyle\sum_{\mathcal{O}\subseteq\mathcal{O}_r}\omega_{\mathcal{O}},$
where $\omega_{\mathcal{O}}\in\Omega^{k-r}(\mathcal{O},o(N^-))^G$ for each critical orbit $\mathcal{O}\subseteq\mathcal{O}_r$. 
Applying (\ref{main body integral but with epsilon adjust by error}), we obtain
\begin{align}\label{tail in main}
&\sum_{\mathcal{O}\subseteq\mathcal{O}_r}e^{Tf(\mathcal{O})}\cdot(\pi_r)_*\left(\left.e^{T(f-f(\mathcal{O}))}\cdot (J_T\omega+\tau)\right\vert_{\overline{W^u(\mathcal{O})}}\right) \nonumber\\
=\ & \sum_{\mathcal{O}\subseteq\mathcal{O}_r}e^{Tf(\mathcal{O})}\cdot \mathcal{N}_T(\omega_{\mathcal{O}}) + \sum_{\mathcal{O}\subseteq\mathcal{O}_r}e^{Tf(\mathcal{O})}\cdot(\pi_r)_*\left(\left.e^{T(f-f(\mathcal{O}))}\cdot \tau\right\vert_{\overline{W^u(\mathcal{O})}}\right) \nonumber\\
& + \sum_{\substack{\mathcal{O}'\subseteq\mathcal{O}_r\\\mathcal{O}'\neq\mathcal{O}}}\sum_{\mathcal{O}\subseteq\mathcal{O}_r}e^{Tf(\mathcal{O}')}\cdot(\pi_r)_*\left(\left.e^{T(f-f(\mathcal{O}'))}\cdot (J_T\omega_{\mathcal{O}}+\tau)\right\vert_{\overline{W^u(\mathcal{O}')}}\right).
\end{align}
By \cite[Lemma 3.3]{Austin1995}, the boundary of $\overline{W^u(\mathcal{O})}$ is given by
\begin{align}
    \bigcup_{\nu = 1}^r\bigcup_{i_0<i_1<\cdots<i_{\nu-1} < r}\mathcal{M}(\mathcal{O},\mathcal{O}_{i_{\nu-1}})\times_{\mathcal{O}_{i_{\nu-1}}}\cdots\times_{\mathcal{O}_{i_1}}\mathcal{M}(\mathcal{O}_{i_1},\mathcal{O}_{i_0})\times_{\mathcal{O}_{i_0}} W^u(\mathcal{O}_{i_0}).
\end{align}
Thus, in (\ref{tail in main}), $\mathcal{O}$ is disjoint from $\overline{W^u(\mathcal{O}')}$. 
Therefore, there is a constant $\xi_{\mathcal{O}'}>0$ such that 
\begin{align}\label{true tail in main}
    \abs{(\pi_r)_*\left(e^{T(f-f(\mathcal{O}'))}\cdot (J_T\omega_{\mathcal{O}}+\tau)\vert_{\overline{W^u(\mathcal{O}')}}\right)}\leqslant e^{-T\xi_{\mathcal{O}'}}
\end{align}
for all $\mathcal{O}'\neq\mathcal{O}$ in (\ref{tail in main}). In addition, we have a constant $\Gamma''>0$ (irrelevant to $\varepsilon$) such that 
\begin{align}
\left|(\pi_r)_*\left(\left.e^{T(f-f(\mathcal{O}))}\cdot \tau\right\vert_{\overline{W^u(\mathcal{O})}}\right)\right|\leqslant
\Gamma''\cdot\varepsilon\cdot T^{-r/2}
\end{align}
in (\ref{tail in main}) when $T$ is sufficiently large. 

\vspace{+1mm}
\noindent \textbf{Part II}: When $i<r$, again by \cite[Lemma 3.3]{Austin1995}, the boundary of $\overline{W^u(\mathcal{O}_i)}$ is equal to 
\begin{align}
    \bigcup_{\nu = 1}^i\bigcup_{i_0<i_1<\cdots<i_\nu = i}\mathcal{M}(\mathcal{O}_{i_\nu},\mathcal{O}_{i_{\nu-1}})\times_{\mathcal{O}_{i_{\nu-1}}}\cdots\times_{\mathcal{O}_{i_1}}\mathcal{M}(\mathcal{O}_{i_1},\mathcal{O}_{i_0})\times_{\mathcal{O}_{i_0}} W^u(\mathcal{O}_{i_0}).
\end{align}
Therefore, $\mathcal{O}_r$ is disjoint from $\overline{W^u(\mathcal{O}_i)}$. Thus, there exists a constant $\xi_{ir}>0$ such that 
\begin{align}\label{upper}
    \abs{(\pi_i)_*\left(\left.e^{T(f-f(\mathcal{O}))}\cdot J_T\omega\right\vert_{\overline{W^u(\mathcal{O})}}\right)}\leqslant e^{-T\xi_{ir}}
\end{align}
for all $\mathcal{O}\subseteq\mathcal{O}_i$ when $i>r$.  

\vspace{+1mm}
\noindent \textbf{Part III}: When $i>r$, we have a constant $\xi_{ir}>0$ such that 
\begin{align}\label{lower}
    \abs{(\pi_i)_*\left(\left.e^{T(f-f(\mathcal{O}))}\cdot J_T\omega\right\vert_{\overline{W^u(\mathcal{O})}}\right)}\leqslant \xi_{ir}
\end{align}
for all $\mathcal{O}\subseteq\mathcal{O}_i$. 

\vspace{+1mm}
Finally, we let $\kappa = \dim C^*(M,f)^G$ and recall that $m = \dim M$. The Gaussian integral (adjusted by $\rho$) in (\ref{main body integral but with epsilon adjust by error}) satisfies
\begin{align}\label{the estimate of gaussian integrals}
    \int_{\mathbb{R}^i}\rho(y_1,\cdots,y_i,0,\cdots,0)e^{-T\left(y_1^2+\cdots+y_i^2\right)}dy_1\cdots dy_i = O(T^{-i/2})
\end{align}
for all $0\leqslant i\leqslant m$ when $T\to +\infty$.
Combining {(\ref{diagonal matrix by function values on critical set})}-{(\ref{the estimate of gaussian integrals})}, 
after selecting and arranging an orthonormal basis of $C^*(M,f)^G$, we find that 
\begin{align}
    \Phi_T P_T(\alpha)J_T = \mathcal{F}_T\circ\mathcal{N}_T\circ (X+Y),
\end{align}
    where $X$ and $Y$ are $\kappa\times\kappa$ matrices such that when $T\to+\infty$, their $(s,t)$-th entries satisfy
    \begin{align}
      \begin{array}{ll}
      X_{st} = 1 + \Gamma''\varepsilon& \text{\ when\ } s=t, \\
      \abs{X_{st}}\leqslant O(T^{m/2}) & \text{\ when\ } s>t, \\
      X_{st} = 0 & \text{\ when\ } s<t,
\end{array} \text{\ and\ } \abs{Y_{st}} \leqslant
      O(T^{-\mu+m/2}) \text{\ for any\ } (s,t).
    \end{align}
    According to the Leibniz formula \cite[(4.16)]{greub1981linear} of determinant, we find that when $T\to+\infty$, 
    \begin{align}
        \det(X+Y) = (1+\Gamma''\varepsilon)^\kappa+O(T^{-\mu+\kappa m/2}).
    \end{align}
Notice that $\Gamma''$ is irrelevant to $\varepsilon$, we can let $\Gamma''\varepsilon\ll 1$. By choosing a sufficiently large $\mu$ in Proposition \ref{key estimate global L infinity}, the map $\Phi_T P_T(\alpha) J_T$ is an isomorphism when $T>0$ is sufficiently large. 

In addition, by Corollary \ref{isomorphism without cochain operations}, $P_T(\alpha) J_T$ is an isomorphism between vector spaces when $T$ is sufficiently large. Therefore, $\Phi_T$ is a cochain isomorphism for sufficiently large $T$. 
\end{proof}

\section{Examples and corollaries}\label{section of computations and examples}
In this section, we first give some examples using the $G$-invariant Thom-Smale complex to compute the Betti numbers of $M$. Next, we apply the analytic realization to deduce the $G$-invariant Morse-Bott inequalities and some facts on eigenspaces. Finally, we end the paper by explaining how we notice the constant $\alpha_0$ in Theorem \ref{isom given under a good metric be careful}. 

\begin{example}\label{example 1 of topological result}
\normalfont
   If $M = G$ acts by the left multiplication, and $f$ is a constant, we have 
   \begin{align}
       C^k(M,f)^G = \Omega^k(G)^G,\  \forall\hspace{+1mm} k\in\mathbb{Z}_{\geqslant 0},
   \end{align}
    and the boundary map $\partial = d$. This repeats the $G$-invariant de Rham complex of $G$. 
\end{example}

\begin{example}\label{example 2 of topological result}
\normalfont
    Let $\mathbb{S}^n$ be the unit $n$-sphere. We let $M = \mathbb{S}^2\times\mathbb{S}^1$ and $G = \mathbb{S}^1$ with the action
    \begin{align}
        & \theta\cdot (x,y,z,t) \ \ (\text{Here,}\ x^2+y^2+z^2 = 1)   \nonumber\\
        =\ &\left(x\cos\theta-y\sin\theta, x\sin\theta+y\cos\theta,z,2\theta+t\right).
    \end{align}
Then, we see that 
\begin{align}
    f(x,y,z,t) =
(x^2-y^2)\cos t+2xy\sin t
\end{align}
is an $\mathbb{S}^1$-invariant function. In addition, $\crit(f)$ consists of the following $\mathcal{O}_0, \mathcal{O}_1$ and $\mathcal{O}_2$: 
\begin{enumerate}[label = (\arabic*)]
\item $\mathcal{O}_0 = 
\left\{(-\sin t,\cos t,0,2t)\in\mathbb{S}^2\times\mathbb{S}^1:t\in\mathbb{R}\right\}$, which is diffeomorphic to the unit circle. %The diffeomorphism 
%\begin{align}
    %\overline{\chi}_0: \mathbb{R}/\mathbb{Z}&\longrightarrow \mathcal{O}_0\\
   % t+\mathbb{Z}&\longmapsto (-\sin(2\pi t),\cos(2\pi t),0,e^{\sqrt{-1}4\pi t})
%\end{align}

\item $\mathcal{O}_1 = \mathcal{O}_{1,1}\cup \mathcal{O}_{1,2}$, where the two orbits are
\begin{align}
    \mathcal{O}_{1,1}& =
\left\{(0,0,1,t)\in\mathbb{S}^2\times\mathbb{S}^1: t\in\mathbb{R}\right\},  \\
\mathcal{O}_{1,2}& = 
\left\{(0,0,-1,t)\in\mathbb{S}^2\times\mathbb{S}^1: t\in\mathbb{R}\right\}.
\end{align}
Both $\mathcal{O}_{1,1}$ and $\mathcal{O}_{1,2}$ are diffeomorphic to the unit circle. 
\item $\mathcal{O}_2\coloneqq
\left\{(\cos t,\sin t,0,2t)\in\mathbb{S}^2\times\mathbb{S}^1: t\in\mathbb{R}\right\}$, which is diffeomorphic to the unit circle. 
\end{enumerate}
Observing that $W^u(\mathcal{O}_0)$ and $W^u(\mathcal{O}_2)$ are orientable, while $W^u(\mathcal{O}_{1,1})$ and $W^u(\mathcal{O}_{1,2})$ are nonorientable, the $G$-invariant Thom-Smale complex is given by
\begin{align}
    0\xrightarrow[]{}\Omega^0(\mathcal{O}_0)^G\xrightarrow[]{0}\Omega^1(\mathcal{O}_0)^G\xrightarrow[]{0}\Omega^0(\mathcal{O}_2)^G\xrightarrow[]{0}\Omega^1(\mathcal{O}_2)^G\xrightarrow[]{}0.
\end{align}
Thus, we find that for $k = 0,1,2,3$, the $k$-th cohomology group of $M$ is $\mathbb{R}$.
This gives the same result as applying the K\"unneth formula \cite[Exercise 4.8]{morita2001geometry}. 
\end{example}

\begin{example}\label{example 3 of topological result}
    \normalfont
   Let $G = \mathbb{T}^2$ be the $2$-torus and $M = \mathbb{S}^2\times\mathbb{T}^2$ admitting the $\mathbb{T}^2$-action
    \begin{align}
        & (\theta_1,\theta_2)\cdot (x,y,z,t,s)    \nonumber\\
        =\ &\left(x\cos(\theta_1+2\theta_2)-y\sin(\theta_1+2\theta_2), x\sin(\theta_1+2\theta_2)+y\cos(\theta_1+2\theta_2),z,t+2\theta_1,s+\theta_2\right).
    \end{align}
Then, we find that 
\begin{align}
    f(x,y,z,t,s) \coloneqq
(x^2-y^2)\cos(t+4s)+2xy\sin(t+4s)
\end{align}
is a $\mathbb{T}^2$-invariant function. Its critical set consists of the following $\mathcal{O}_0, \mathcal{O}_1$ and $\mathcal{O}_2$: 
\begin{enumerate}[label = (\arabic*)]
\item $\mathcal{O}_0=
\left\{(-\sin t, \cos t,0,2t-4s,
s)\in\mathbb{S}^2\times\mathbb{T}^2:t,s\in\mathbb{R}\right\}$, which is diffeomorphic to $\mathbb{T}^2$. 
\item $\mathcal{O}_1=\mathcal{O}_{1,1}\cup \mathcal{O}_{1,2}$, where the two critical orbits are
\begin{align}
    \mathcal{O}_{1,1}&=
\left\{(0,0,1,t,s)\in\mathbb{S}^2\times\mathbb{T}^2:t,s\in\mathbb{R}\right\},\\
\mathcal{O}_{1,2}&=
\left\{(0,0,-1,t,s)\in\mathbb{S}^2\times\mathbb{T}^2:t,s\in\mathbb{R}\right\}.
\end{align}
Both $\mathcal{O}_{1,1}$ and $\mathcal{O}_{1,2}$ are diffeomorphic to $\mathbb{T}^2$. 
\item $\mathcal{O}_2 = 
\left\{(\cos t,\sin t,0,2t-4s,s)\in\mathbb{S}^2\times\mathbb{T}^2:t,s\in\mathbb{R}\right\}$, 
which is diffeomorphic to $\mathbb{T}^2$. 
\end{enumerate}
Since $W^u(\mathcal{O}_0)$ and $W^u(\mathcal{O}_2)$ are orientable, while the unstable manifolds of $W^u(\mathcal{O}_{1,1})$ and $W^u(\mathcal{O}_{1,2})$ are nonorientable, we get the $G$-invariant Thom-Smale complex
\begin{align}
    0\xrightarrow[]{}\Omega^0(\mathcal{O}_0)^G\xrightarrow[]{0}\Omega^1(\mathcal{O}_0)^G\xrightarrow[]{0}\Omega^0(\mathcal{O}_2)^G\oplus\Omega^2(\mathcal{O}_0)^G\xrightarrow[]{0}\Omega^1(\mathcal{O}_2)^G\xrightarrow[]{0}\Omega^2(\mathcal{O}_2)^G\xrightarrow[]{}0.
\end{align}
Thus, we find that for $k = 0,4$, the $k$-th cohomology group is $\mathbb{R}$, and for $k = 1,2,3$, 
the group is $\mathbb{R}\oplus\mathbb{R}$. 
This gives the same result as applying the K\"unneth formula. 
\end{example}
\begin{example}\label{example 4 of topological result}
\normalfont
    We define the action of $G = \mathbb{S}^1$ on $M = \mathbb{S}^3$: 
    \begin{align}
           & \theta\cdot (x,y,z,w)  \ \ (\text{Here,}\ x^2+y^2+z^2+w^2 = 1)   \nonumber\\
           =\ &(x\cos(2\theta)-y\sin(2\theta), x\sin(2\theta)+y\cos(2\theta), z\cos\theta-w\sin\theta, z\sin\theta+w\cos\theta).
    \end{align}
    Then, we find the function
    \begin{align}
        f(x,y,z,w) = (z^2-w^2)x+2zwy
    \end{align}
    is an $\mathbb{S}^1$-invariant Morse-Bott function, of which the critical set consists of
    \begin{align}
        \mathcal{O}_0& = \left\{\left(-\dfrac{\sqrt{3}}{3}\cos(2t),\dfrac{\sqrt{3}}{3}\sin(2t),\dfrac{2\sqrt{3}}{3}\cos t,\dfrac{2\sqrt{3}}{3}\sin t\right)\in\mathbb{S}^3: t\in\mathbb{R}\right\},\\
    \mathcal{O}_1& = \left\{(\cos t,\sin t,0,0)\in\mathbb{S}^3: t\in\mathbb{R}\right\},\\
    \mathcal{O}_2& = \left\{\left(\dfrac{\sqrt{3}}{3}\cos(2t),\dfrac{\sqrt{3}}{3}\sin(2t),\dfrac{2\sqrt{3}}{3}\cos t,\dfrac{2\sqrt{3}}{3}\sin t\right)\in\mathbb{S}^3: t\in\mathbb{R}\right\}.
    \end{align}
    All of them are diffeomorphic to $\mathbb{S}^1$. 
    Since $W^u(\mathcal{O}_1)$ is nonorientable, the $G$-invariant Thom-Smale complex is 
    \begin{align}
    0\xrightarrow[]{}\Omega^0(\mathcal{O}_0)^G\xrightarrow[]{0}\Omega^1(\mathcal{O}_0)^G\xrightarrow[]{\omega\mapsto\int_0^{2\pi}\omega}\Omega^0(\mathcal{O}_2)^G\xrightarrow[]{0}\Omega^1(\mathcal{O}_2)^G\xrightarrow[]{}0.
\end{align}
Therefore, we find for $k = 0,3$, the $k$-th cohomology group is $\mathbb{R}$, and for $k = 1,2$, the group is $0$. This is exactly the de Rham cohomology of $\mathbb{S}^3$. 
\end{example}

We now give some corollaries of the analytic result Theorem \ref{isom given under a good metric be careful}. A straightforward one counts the number of eigenvalues of $D_T^2$ on $\Omega^*(M)^G$.

\begin{corollary}\label{application 1 of analytic result}
    For any constant $\alpha>\alpha_0$, when $T$ is sufficiently large, the number of eigenvalues $\leqslant \alpha$ of $D_T^2\vert_{\Omega^*(M)^G}$ is equal to $\dim C^*(M,f)^G$. 
\end{corollary}
\begin{proof}
    This is given by the definition of $F_T^*(M,f,\alpha)^G$.
\end{proof}

In addition, since both $C^*(M,f)^G$ and $F_T^*(M,f,\alpha)^G$ are finite dimensional, we obtain simplified proofs of the weak and strong $G$-invariant Morse-Bott inequalities associated with our $f$. We recall some notations:
\begin{enumerate}[label = (\arabic*)]
    \item\ $m = \dim M$;
    \item\ $H^k(M)^G$ is the $k$-th $G$-invariant de Rham cohomology group of $M$;
    \item\ $H^j(\mathcal{O}_i,\mathcal{E}_i)^G$ is the $j$-th $G$-invariant de Rham cohomology group of $\mathcal{O}_i$ with local coefficient $\mathcal{E}_i$ (See (\ref{de Rham twisted with coefficients})).
\end{enumerate}
The weak version is as follows. 
\begin{corollary}\label{application 2 of analytic result}
   For any $0\leqslant k\leqslant m$, we have
   $\dim C^k(M,f)^G\geqslant \dim H^k(M)^G$. 
\end{corollary}
\begin{proof}
    Since $H^k(M)^G$ is isomorphic to the kernel of $D_T^2$ restricted on $\Omega^k(M)^G$, we find
    \begin{align}
        \dim C^k(M,f)^G =\ & \dim F_T^k(M,f,\alpha)^G \nonumber\\
        \geqslant\ & \dim\ker\left(D_T^2\vert_{\Omega^k(M)^G}\right) \nonumber\\
        =\ & \dim H^k(M)^G
    \end{align}
    and get the weak inequalities. 
\end{proof}

We then prove the following strong version. As \cite[Corollary 3.9]{Austin1995}, we can prove the strong version only using Theorem \ref{thom smale cohomology theorem}. However, here we give a proof using Theorem \ref{isom given under a good metric be careful}.
\begin{corollary}\label{morse bott inequalities need proof}
     For any $0\leqslant k\leqslant m$, we have 
    \begin{align}\label{inequalities strong expression}
        \sum_{r = 0}^k\sum_{i+j = r} (-1)^{k-r} \dim H^j(\mathcal{O}_i,\mathcal{E}_i)^G \geqslant \sum_{r = 0}^k (-1)^{k-r}\dim H^r(M)^G
    \end{align}
    and call them the $G$-invariant Morse-Bott inequalities associated with $f$. 
\end{corollary}
\begin{proof}
    Applying the dimension formula to the Witten instanton complex, we have 
    \begin{align}\label{left side ineq}
 \dim F_T^k(M,f,\alpha)
    =\ &\dim\ker\left(d_T\vert_{F_T^k(M,f,\alpha)^G}\right)+\dim\image\left(d_T\vert_{F_T^k(M,f,\alpha)^G}\right) \nonumber\\
     =\ & \dim H^k(M)^G+\dim\image\left(d_T\vert_{F_T^k(M,f,\alpha)^G}\right) + \dim\image\left(d_T\vert_{F_T^{k-1}(M,f,\alpha)^G}\right). 
    \end{align}
    For the same reason on the Thom-Smale complex, we find
    \begin{align}\label{right side ineq}
        \dim C^k(M,f)^G 
        =\ & \sum_{i+j = k}\dim \Omega^j(\mathcal{O}_i,\mathcal{E}_i)^G \nonumber\\
     =\ & \sum_{i+j = k} \dim\ker \left(d\vert_{\Omega^j(\mathcal{O}_i,\mathcal{E}_i)^G}\right) + \dim\image \left(d\vert_{\Omega^j(\mathcal{O}_i,\mathcal{E}_i)^G}\right) \nonumber\\
        =\ & \sum_{i+j = k} \dim H^j(\mathcal{O}_i,\mathcal{E}_i)^G + \dim\image \left(d\vert_{\Omega^j(\mathcal{O}_i,\mathcal{E}_i)^G}\right) + \dim\image \left(d\vert_{\Omega^{j-1}(\mathcal{O}_i,\mathcal{E}_i)^G}\right). 
    \end{align}
    By (\ref{left side ineq}) and (\ref{right side ineq}), we obtain 
    \begin{align}\label{left ineq 11}
        & \sum_{r = 0}^k (-1)^{k-r}\dim F_T^r(M,f,\alpha)^G \nonumber\\
        =\ & \sum_{r = 0}^k (-1)^{k-r}\left(\dim H^r(M)^G+\dim\image\left(d_T\vert_{F_T^r(M,f,\alpha)^G}\right) + \dim\image\left(d_T\vert_{F_T^{r-1}(M,f,\alpha)^G}\right)\right)\nonumber\\
        =\ & \dim\image\left(d_T\vert_{F^k_T(M,f,\alpha)^G}\right) + \sum_{r = 0}^k (-1)^{k-r}\dim H^r(M)^G
    \end{align}
    and then 
    \begin{align}\label{right ineq 22}
       & \sum_{r = 0}^k (-1)^{k-r}\dim C^r(M,f)^G\nonumber\\
       =\ & \sum_{r = 0}^k \sum_{i+j = r} (-1)^{k-r}\left(\dim\ker H^j(\mathcal{O}_i,\mathcal{E}_i)^G + \dim\image \left(d\vert_{\Omega^j(\mathcal{O}_i,\mathcal{E}_i)^G}\right) + \dim\image \left(d\vert_{\Omega^{j-1}(\mathcal{O}_i,\mathcal{E}_i)^G}\right)\right)\nonumber\\
       =\ &\sum_{i+j = k}\dim\image\left(d\vert_{\Omega^{j}(\mathcal{O}_i,\mathcal{E}_i)^G}\right) + \sum_{r = 0}^k \sum_{i+j = r}(-1)^{k-r} \dim H^j(\mathcal{O}_i,\mathcal{E}_i)^G.
    \end{align}
    By Theorem \ref{isom given under a good metric be careful}, (\ref{left ineq 11}), and (\ref{right ineq 22}), we get
    \begin{align}\label{final ineq}
         & \sum_{r = 0}^k \sum_{i+j = r}(-1)^{k-r} \dim H^j(\mathcal{O}_i,\mathcal{E}_i)^G  - \sum_{r = 0}^k (-1)^{k-r}\dim H^r(M)^G \nonumber\\
         =\ & \dim\image\left(d_T\vert_{F^k_T(M,f,\alpha)^G}\right) - \sum_{i+j = k}\dim\image\left(d\vert_{\Omega^{j}(\mathcal{O}_i,\mathcal{E}_i)^G}\right) \nonumber\\
         =\ & \dim\image\left(\partial\vert_{C^k(M,f)^G}\right) - \sum_{i+j = k}\dim\image\left(d\vert_{\Omega^{j}(\mathcal{O}_i,\mathcal{E}_i)^G}\right).
    \end{align}
    Since $C^*(M,f)^G$ is finite dimensional, for each $0\leqslant j\leqslant k$, we choose an independent subset 
    \begin{align}
        \{\omega^j_{1},\cdots,\omega^j_{s}\}\subseteq \Omega^j(\mathcal{O}_{k-j},\mathcal{E}_i)^G
    \end{align}
    such that $\{d\omega^j_1,\cdots, d\omega^j_s\}$ is a basis of $\image\left(d\vert_{\Omega^{j}(\mathcal{O}_i,\mathcal{E}_i)^G}\right)$. 
    By the definition of $\partial$, the linear map 
    \begin{align}
        \image\left(d\vert_{\Omega^{0}(\mathcal{O}_k,\mathcal{E}_k)^G}\right)\oplus\image\left(d\vert_{\Omega^{0}(\mathcal{O}_k,\mathcal{E}_k)^G}\right)\cdots\oplus\image\left(d\vert_{\Omega^{k}(\mathcal{O}_0,\mathcal{E}_0)^G}\right) & \to \image\left(\partial\vert_{C^k(M,f)^G}\right) \nonumber\\
        (d\omega^0_{s_0},d\omega^1_{s_1},\cdots,d\omega^k_{s_k}) &\mapsto \partial\omega^0_{s_0}+\partial\omega^1_{s_1}+\cdots+\partial\omega^k_{s_k}
    \end{align}
    is injective (but not canonical). We finish the prove by noticing that (\ref{final ineq}) is nonnegative. 
\end{proof}

The proof of Corollary \ref{morse bott inequalities need proof} is purely algebraic. However, it does not intrinsically explain why we involve the cohomology groups of the critical set. 
As in \cite{bismutcomplicated}, \cite{bismutandlebeau}, and \cite{wenlumorsebottineq}, the intrinsic reason is, the kernel of the twisted $(d+d^*)^2\vert_{C^*(M,f)^G}$ corresponds to the sum of eigenspaces of $D_T^2$ associated with sufficiently small eigenvalues. 
More precisely, following the notations in Proposition \ref{spectral gap of the Witten Laplacian proposition}, we let 
\begin{align}
    \xi(\varepsilon,T) = \left(C_0 e^{-C_1T} + \Gamma\varepsilon\right)^2
\end{align}
and give Corollary \ref{application 1 of analytic result} a refinement which is similar to \cite[(2.68)]{wenlumorsebottineq}:  
\begin{corollary}\label{refinement of main result 2}
     For each $0\leqslant k\leqslant m$, when $T$ is sufficiently large, the map $P_T(\xi(\varepsilon,T))\circ J_T$ is an isomorphism between $\ker\left((d+d^*)^2\vert_{C^k(M,f)^G}\right)$ and 
    $F_T^k\left(M,f,\xi(\varepsilon,T)\right)$. 
\end{corollary}
\begin{proof}
    We replace the space $E_T$ in (\ref{definition of J_T will be used later}) by the image of 
    \begin{align}
        \ker\left((d+d^*)^2\vert_{C^k(M,f)^G}\right)
    \end{align}
    under $J_T$. Then, we follow the same analysis as either our previous Sections \ref{section of spectral gaps}-\ref{section of cochain isomorphisms} or \cite[VIII-X]{bismutandlebeau}. 
\end{proof}

In fact, using Corollary \ref{refinement of main result 2} and the isomorphism 
\begin{align}
    \ker\left((d+d^*)^2\vert_{\Omega^j(\mathcal{O}_i,\mathcal{E}_i)^G}\right)\cong H^j(\mathcal{O}_i,\mathcal{E}_i)^G,
\end{align}
we have another proof of Corollary \ref{morse bott inequalities need proof} and reveal a more intrinsic relation between the cohomology groups in the inequalities (\ref{inequalities strong expression}). In fact, this is exactly the spirit of the analysis on the Morse-Bott inequalities associated with more general Morse-Bott functions in \cite{bismutcomplicated} and \cite{wenlumorsebottineq}. 

\begin{remark}\normalfont
    We can even refine Corollary \ref{refinement of main result 2} more to correspond each eigenvalue of the twisted $(d+d^*)^2\vert_{C^k(M,f)^G}$ with the associated subspace of $F_T^k(M,f,\alpha)$. This refined correspondence is the asymptotic behavior of eigenvalues of $D_T^2$ first studied by Helffer and Sj\"ostrand using semi-classical analysis tools in \cite{helfferseigenvalue}. 
\end{remark}

We end this paper by explaining how we notice $\alpha_0$, the spectral radius of $(d+d^*)^2\vert_{C^*(M,f)^G}$. 
\begin{example}\label{TX and MN}
\normalfont
    Let $M = G = SU(2)$ and use the left multiplication action. The basis \cite[Example 3.27]{hall2015lie}
    \begin{align}
        \mathbf{v}_1 = \begin{bmatrix}
            0 & \sqrt{-1}\ \\
            \sqrt{-1} & 0\ 
        \end{bmatrix}, \mathbf{v}_2 = \begin{bmatrix}
            0 & -1\\
            1 & 0
        \end{bmatrix}, \mathbf{v}_3 = \begin{bmatrix}
            \sqrt{-1} & 0\\
             0 & -\sqrt{-1} 
        \end{bmatrix}
    \end{align}
    of $\mathfrak{su}(2)$ generates a left invariant frame $X_1, X_2, X_3$ on $SU(2)$, and we let $X_1, X_2, X_3$ be orthonormal.
    When $f = 0$ on $SU(2)$, we find that each dual 1-form $X_j^*$ satisfies 
    \begin{align}
        D_T^2(X_j^*) = 4 X_j^*,\ j = 1,2,3.
    \end{align}
Therefore, compared with the Witten instanton complex in \cite{Bismut1994} and \cite{wittendeformationweipingzhang} where the number $\alpha$ in $F^k_T(M,f,\alpha)^G$ can be arbitrarily small, we need a nontrivial lower bound (which is $\alpha_0$) of $\alpha$ because of those eigenvalues of $D_T^2$ coming from the horizontal direction. 
\end{example}
The next example shows that $\alpha_0$ relies on both $M$ and $G$ instead of only on $G$. 
\begin{example}\label{irrelevant to group actions}
\normalfont
    We let $G = SO(3)$, $M = \mathbb{S}^2$, and $f = 0$ on $M$. Then, $G$ acts on $M$ transitively. Equipping $M$ with a $G$-invariant metric induced by $G$ as in Section \ref{sections of metrics and connections}, we let $\dvol_M$ be the unit volume form with respect to this metric and find: 
    \begin{enumerate}[label = (\arabic*)]
        \item  $\Omega^0(M)^G = \mathbb{R}$.
        \item  $\Omega^1(M)^G = \{0\}$. This is because the dual of a $G$-invariant $1$-form is either $0$ or a nonvanishing vector field. The latter is impossible on $\mathbb{S}^2$. 
        \item  $\Omega^2(M)^G = \mathbb{R}\cdot\dvol_M$. 
    \end{enumerate}
    Thus, in this case, $\alpha_0$ should be $0$. However, noticing the isomorphism $\mathfrak{so}(3)\cong \mathfrak{su}(2)$ between Lie algebras, if we only consider $SO(3)$ instead of considering $SO(3)$ and $\mathbb{S}^2$ together when determining $\alpha_0$, we will get $\alpha_0\geqslant 4$ by Example \ref{TX and MN}, which is not the minimal choice. 
\end{example}

\bibliographystyle{abbrv}
\bibliography{reference.bib}
\end{document}